\documentclass[11pt]{amsart}

\usepackage{amsmath, amsfonts, amssymb, mathrsfs}

\newcommand{\cA}{ {\mathcal A} }
\newcommand{\cB}{ {\mathcal B} }
\newcommand{\cC}{ {\mathcal C} }
\newcommand{\bC}{ {\mathbb C} }
\newcommand{\bN}{ {\mathbb N} }
\newcommand{\bQ}{ {\mathbb Q} }
\newcommand{\bR}{ {\mathbb R} }

\newcommand{\cT}{{\mathcal T}}
\newcommand{\cX}{{\mathcal X}}
\newcommand{\bZ}{ {\mathbb Z} }

\newcommand{\FF}{ \chi }
\newcommand{\chitild}{ \widetilde{\chi} }
\newcommand{\Kr}{ {\mathrm{Kr}} }
\newcommand{\moeb}{ {\mathrm{Moeb}} }
\newcommand{\sign}{ {\mathrm{sign}} }
\newcommand{\switch}{ {\mathrm{Switch}} }
\newcommand{\term}{ {\mathrm{term}} }
\newcommand{\Term}{ {\mathrm{Term}} }
\newcommand{\deltabar}{ \overline{\delta} }
\newcommand{\unitA}{ 1_{ { }_{\cA} } }
\newcommand{\ee}{ \varepsilon }

\newcommand{\kpieven}{ \Kr ( \pi )^{\mathrm{(even)}} } 
\newcommand{\roeven}{ \rho^{\mathrm{(even)}} } 
 
\newcommand{\qpoints}{ {\mathrm{(q-points)}} } 
\newcommand{\rhofateven}{ \rho^{\mathrm{(q-points)}} } 
\newcommand{\rhoqpairing}{ \rho^{\mathrm{(q-pairing)}} }

\newcommand{\odd}{ {\mathrm{odd}} }
\newcommand{\piodd}{ \pi^{ ( \odd ) } } 

\newcommand{\pifatodd}{ \pi^{\mathrm{(u-points)}} }

\newtheorem{theorem}{Theorem}[section]
\newtheorem{lemma}[theorem]{Lemma}
\newtheorem{proposition}[theorem]{Proposition}
\newtheorem{prop-and-notation}[theorem]{Proposition and Notation}
\newtheorem{corollary}[theorem]{Corollary}

\theoremstyle{remark}
\newtheorem{remark}[theorem]{Remark}
\newtheorem{definition}[theorem]{Definition}
\newtheorem{notation}[theorem]{Notation}
\newtheorem{def-and-remark}[theorem]{Definition and Remark}
\newtheorem{rem-and-notation}[theorem]{Remark and Notation}
\newtheorem{example}[theorem]{Example}

\numberwithin{equation}{section}

\title{Star-cumulants of free unitary Brownian motion}

\author[N. Demni]{Nizar Demni}
\thanks{Nizar Demni was supported in part by Grant
ANR-11-LABEX-0020-01 from ANR, France.}
\address{N. Demni: IRMAR, Universit\'e de Rennes 1  \\
Campus de Beaulieu    \\ 
35042 Rennes cedex, France}
\email{nizar.demni@univ-rennes1.fr}
\author[M. Guay-Paquet]{Mathieu Guay-Paquet}
\thanks{Mathieu Guay-Paquet was supported by a Postdoctoral
Fellowship from NSERC, Canada.}
\address{M. Guay-Paquet: LaCIM   \\
Universit\'e du Qu\'ebec \`a Montr\'eal  \\
201 Pr\'esident-Kennedy   \\
Montr\'eal, Qu\'ebec H2X 3Y7, Canada}
\email{mathieu.guaypaquet@lacim.ca}
\author[A. Nica]{Alexandru Nica}
\thanks{Alexandru Nica was supported in part by a Discovery 
Grant from NSERC, Canada.}
\address{A. Nica: Department of Pure Mathematics    \\
University of Waterloo            \\
Waterloo, Ontario N2L 3G1, Canada}
\email{anica@uwaterloo.ca}

\begin{document}

\begin{abstract}
We study joint free cumulants of $u_t$ and $u_t^{*}$, where 
$u_t$ is a free unitary Brownian  motion at time $t$.  We 
determine explicitly some special families of such cumulants.
On the other hand, for a general joint cumulant of $u_t$ and 
$u_t^{*}$, we ``calculate the derivative'' for $t \to \infty$,
when $u_t$ approaches a Haar unitary.  In connection to the 
latter calculation we put into evidence an ``infinitesimal 
determining sequence'' which naturally accompanies an arbitrary
$R$-diagonal element in a tracial $*$-probability space.
\end{abstract}

\maketitle

\section{Introduction}

Let $(u_t)_{t \geq 0}$ be a free unitary Brownian motion 
in the sense of \cite{BV1992}, \cite{B1997} --- that is,
every $u_t$ is a unitary element in some 
tracial $*$-probability space $(\cA_t , \varphi_t )$, with 
$\varphi_t (u_t ) = e^{-t/2}$, and where the rescaled element 
$v_t := e^{t/2} u_t$ has $S$-transform given by
\begin{equation}  \label{eqn:1a}
S_{ v_t } (z) = e^{tz}, \ \ z \in \bC.
\end{equation}
Closely related to Equation (\ref{eqn:1a}), one has a nice 
formula for the free cumulants of $u_t$, i.e.~for the 
sequence of numbers 
$\bigl( \, \kappa_n (u_t, \ldots , u_t) \, \bigr)_{n=1}^{\infty}$, 
where $\kappa_n : \cA_t^n \to \bC$ is the $n$-th free 
cumulant functional of the space $( \cA_t , \varphi_t )$.  Indeed, 
these numbers are the coefficients of the $R$-transform $R_{u_t}$.
By using the relation between the $S$-transform and the 
compositional inverse of the $R$-transform (which simply says that 
$zS(z) = R^{ <-1> } (z)$), one finds that 
\begin{equation}  \label{eqn:1b}
R_{v_t} (z) = \frac{1}{t} W(tz), \ \ t > 0,
\end{equation}
where 
\begin{equation*}
W (y) = y - y^2 + \frac{3}{2} y^3 - \frac{8}{3} y^4 
+ \cdots + \frac{(-n)^{n-1}}{n!} y^n + \cdots
\end{equation*}
is the Lambert series.  Extracting the coefficient of $z^n$ in 
(\ref{eqn:1b}) gives the value of $\kappa_n (v_t, \ldots , v_t)$, 
then rescaling back gives
\begin{equation}  \label{eqn:1c}
\kappa_n ( u_t, \ldots , u_t ) 
= e^{-nt/2} \kappa_n ( v_t, \ldots , v_t ) 
= e^{-nt/2} \frac{(-n)^{n-1}}{n!}  \cdot t^{n-1}, 
\, \, n \in \bN, \, t \geq 0.
\end{equation}

In this paper we study joint free cumulants of $u_t$ and 
$u_t^{*}$, that is, quantities of the form
\[
\kappa_n \bigl( \, u_t^{ \omega (1) }, \ldots , 
u_t^{ \omega (n) } \, \bigr), 
\mbox{ where $n \in \bN$ and }
\omega = ( \omega (1), \ldots , \omega (n))\in \{ 1, * \}^n.  
\]
The motivation for paying attention to these joint free cumulants 
comes from looking at the limit $t \to \infty$, when $u_t$
approximates in distribution a Haar unitary.  Recall that a unitary 
$u$ in a $*$-probability space $(\cA , \varphi )$ is said to be a 
Haar unitary when it has the property that $\varphi (u^n) = 0$ for 
every $n \in \bZ \setminus \{ 0 \}$.  This property trivially implies 
$\kappa_n (u, \ldots , u) = 0$ for every $n \in \bN$, thus the free 
cumulants of $u$ alone do not look too exciting.  However, things 
become interesting upon considering the larger family of joint free 
cumulants of $u$ and $u^{*}$.  There we get the following 
non-trivial formula, first found in \cite{S1998}: for $\omega$ =
$( \omega (1), \ldots , \omega (n) ) \in \{ 1, * \}^n$ one has
\begin{equation}  \label{eqn:1e}
\kappa_n \bigl( u^{\omega (1)}, \ldots , u^{\omega (n)} \, \bigr) 
= \left\{  \begin{array}{cl}
(-1)^{k-1} C_{k-1},
        &    \mbox{if $n$ is even, $n = 2k$, and}                   \\
        &    \mbox{ $\ $    } \omega = (1,*,1,*, \ldots ,1,*)       \\
        &    \mbox{ $\ $ or } \omega = (*,1,*,1, \ldots ,*,1),      \\
0,      &    \mbox{otherwise,}
\end{array}  \right.
\end{equation} 
with $C_{k-1} = (2k-2)! / (k-1)!k!$, the $(k-1)$-th Catalan 
number.  Formula (\ref{eqn:1e}) leads to the combinatorial approach 
to $R$-diagonal elements --- these are elements in a $*$-probability 
space which display, in some sense, free independence in their polar 
decomposition (\cite{NS1997}, see also Lecture 15 of \cite{NS2006};
a brief review appears in Remark \ref{rem:2.6} below).
For an $R$-diagonal element $a$ in a tracial $*$-probability space
$( \cA , \varphi )$, the sequence 
$\bigl( \, \kappa_{2n} 
(a,a^{*}, \ldots , a,a^{*}) \, \bigr)_{n=1}^{\infty}$
is called ``determining sequence of $a$'' (and does indeed 
determine the joint distribution of $a$ and $a^{*}$); from this point 
of view, Equation (\ref{eqn:1e}) says that the determining sequence
of a Haar unitary consists of signed Catalan numbers.

Returning to the point of view that, for $t \to \infty$, the free 
unitary Brownian motion $u_t$ is an approximation of $u$, it then 
becomes natural to ask what can be said about the joint free 
cumulants of $u_t$ and $u_t^{*}$.  The expressions for these joint 
cumulants are far more involved than what is on the right-hand side 
of (\ref{eqn:1e}), but still turn out to have some tractable features.  
In order to discuss them, it is convenient to start from the fact 
(easily obtained from the general formula connecting free cumulants 
to moments in a noncommutative probability space) that for every 
fixed $\omega \in \{ 1, * \}^n$, the cumulant  $\kappa_n \bigl( 
\, u_t^{\omega (1)}, \ldots , u_t^{\omega (n)}  \, \bigr)$ is a 
quasi-polynomial in $-t/2$; more precisely, there exists a polynomial 
$Z_{\omega} \in \bQ [x,y]$, uniquely determined, such that 
\begin{equation}   \label{eqn:1f}
\kappa_n \bigl( \, u_t^{\omega (1)}, \ldots , 
                   u_t^{\omega (n)}  \, \bigr)
= Z_{\omega} (t, e^{-t/2} ), \ \ \forall \, t \in [ 0, \infty ).
\end{equation}
It is moreover easy to see that the $y$-degree of $Z_{\omega} (x,y)$ 
is at most $n$, and that all the powers of $y$ that appear in 
$Z_{\omega}$ have exponents of same parity as $n$.  In other words, 
we can write
\begin{equation}   \label{eqn:1g}
Z_{\omega} (x,y) = Z_{\omega}^{(n)} (x) \cdot y^n 
+ Z_{\omega}^{(n-2)} (x) \cdot y^{n-2} + \cdots ,
\end{equation}
with $Z_{\omega}^{(n)}, Z_{\omega}^{(n-2)}, \ldots$ in $\bQ [x]$. 
Note that the formula (\ref{eqn:1c}) which describes the cumulants 
of $u_t$ (without $u_t^{*}$) fits here, and can be read as 
\begin{equation}   \label{eqn:1gg}
Z_{ ( \, \underbrace{1,1, \ldots , 1}_n \, ) } (x,y) 
= \frac{(-n)^{n-1}}{n!}  \, x^{n-1} y^n, \ \ n \in \bN .
\end{equation}  

A less obvious fact about the polynomials $Z_{\omega}$ is that 
the number of relevant terms (counting from the top) in the 
expansion (\ref{eqn:1g}) is limited by how many times one switches
between the symbols `$1$' and `$*$', while going around the string
$\omega$.  Thus for $\omega = (1,1, \ldots , 1)$ we have 
$Z_{\omega} (x,y) = Z_{\omega}^{(n)} (x) \cdot y^n$
(as just seen above), then for $\omega$ of the form 
$(1,\ldots , 1,*, \ldots , *)$ we have
$Z_{\omega} (x,y) = Z_{\omega}^{(n)} (x) \cdot y^n
+ Z_{\omega}^{(n-2)} (x) \cdot y^{n-2}$, and so on.  This fact is
stated precisely in Section 3 of the paper, and proved in 
Theorem \ref{thm:3.7} of that section.  It is significant because
it gives information on the speed of decay of $\kappa_n 
\bigl( \, u_t^{\omega (1)}, \ldots , u_t^{\omega (n)}  \, \bigr)$
when $t \to \infty$, in the case (covering ``most'' strings 
$\omega \in \{ 1,* \}^n$) when the right-hand side of 
Equation (\ref{eqn:1e}) is equal to $0$.

Here are some more details about what we do in this paper, and 
about how it is organized.  Besides the present introduction, 
we have five sections.  After a review of background and 
notations in Section 2, some general basic properties of the 
polynomials $Z_{\omega}$ are established in Section 3.  Then 
Sections 4 and 5 study two special types of $\omega$'s, as follows.

$\bullet$ In Section 4 we look at strings of the form
$\omega = (1, \ldots , 1, *, \ldots , *)$, with $k$ occurrences
of `$1$' followed by $\ell$ occurrences of `$*$'. 
We retrieve by direct calculation the fact mentioned above, that
the expansion from Equation (\ref{eqn:1g}) is in this case reduced 
to its top two terms, and we show moreover how the two 
polynomials $Z_{\omega}^{(k+ \ell) } (x)$ and
$Z_{\omega}^{(k + \ell -2)} (x)$ can be written 
explicitly as Laplace transform integrals.

$\bullet$ In Section 5 we look at the case when $\omega$ is an 
alternating string of even length; in other words, we pay attention
(as suggested by formula (\ref{eqn:1e})) to free cumulants 
\[
\xi_n (t) := \kappa_{2n} ( u_t, u_t^{*}, \ldots , u_t, u_t^{*} ), 
\ \ n \in \bN, \ t \in [ 0, \infty ).
\]
The main point of this section is to observe a recursive formula 
for $\frac{d}{dt} \xi_n (t)$, which amounts to the fact that the 
generating function 
\[
H(t,z) := \frac{1}{2} + \sum_{n=1}^{\infty} \xi_n (t) z^n
\]
satisfies a quasi-linear partial differential equation 
of Burgers type,
\[
\partial_t H + 2z H \, \partial_z H = z, 
\ \mbox{ with initial condition } H(0,z) = 1/2.
\]
We also show how examining the 
characteristic curves of the above partial differential equation
gives further information on $\xi_n (t)$.

\vspace{6pt}

Finally, in Section 6 we look at a general string $\omega$, and 
we study the behaviour of the corresponding joint cumulant of 
$u_t$ and $u_t^{*}$ when $t \to \infty$.  We look at the limit
\[
\lim_{t \to \infty}  \frac{ 
\kappa_n \bigl( u_t^{ \omega (1) }, \ldots ,
                u_t^{ \omega (n) }  \bigr)     -
\kappa_n \bigl( u^{ \omega (1) }, \ldots ,
                u^{ \omega (n) }  \bigr)    }{e^{-t/2}} ,
\]
where $u$ is a Haar unitary.  This limit turns out to always 
exist, and to have a very pleasing form, which suggests some kind 
of ``infinitesimal determining sequence'' for a Haar unitary.  In 
Section 6 we also show how the idea of infinitesimal determining 
sequence can be extended to the framework of a general $R$-diagonal 
distribution --- this is done by considering products $u_t q$ where 
$q = q^{*}$ is free from $\{ u_t, u_t^{*} \}$, and then by taking 
the same kind of ``derivative at $t = \infty$'' as above.

$\ $

$\ $

\section{Background and Notation}

\setcounter{equation}{0}

This section gives a very concise review, intended mostly 
for setting notations, of free cumulants on a noncommutative 
probability space.  We follow the terminology from \cite{NS2006}
and, for the various definitions and facts stated below, we give 
specific page references to that monograph.

We start with the structure lying at the basis of the combinatorics
of free probability, the lattices $NC(n)$ of non-crossing partitions.
We will assume the reader to be familiar with these objects, and we 
merely list below some basic notation that we will use in connection 
to them.

$\ $

\begin{notation}                 \label{def:2.1}
{\em [$NC(n)$-terminology.] }
Let $n$ be a positive integer, and let us consider the set $NC(n)$ 
of all non-crossing partitions of $\{ 1, \ldots , n \}$.  

\vspace{6pt}

$1^o$ Partitions in $NC(n)$ will be denoted by letters like 
$\pi,\rho , \ldots$.  Typical explicit notation for a 
$\pi \in NC(n)$ is $\pi = \{ V_1 , \ldots , V_k \}$,
where $V_1 , \ldots , V_k$ are called the {\em blocks} of $\pi$. 
We sometimes simply write $V \in \pi$ to mean that 
``$V$ is one of the blocks of $\pi$.''

\vspace{6pt}

$2^o$ On $NC(n)$ we consider the partial order given by 
{\em reverse refinement}, where for $\pi , \rho \in NC(n)$ we have
$\pi \leq \rho$ if and only if every block of $\rho$ is a union of
blocks of $\pi$.  The partially ordered set $( NC(n), \leq )$ turns 
out to be a {\em lattice} --- that is, every $\pi , \rho \in NC(n)$ 
have a smallest common upper bound $\pi \vee \rho$ and a greatest 
common lower bound $\pi \wedge \rho$.  
(See \cite{NS2006}, pp.~144-146.)

The minimal and maximal element of $( NC(n), \leq )$ are denoted as 
$0_n$ (the partition of $\{ 1, \ldots , n \}$ into $n$ blocks of $1$ 
element each) and respectively as $1_n$ (the partition of 
$\{ 1, \ldots , n \}$ into one block with $n$ elements).

\vspace{6pt}

$3^o$  $( NC(n), \leq )$ has a special anti-automorphism called the 
{\em Kreweras complementation map}, which will be denoted as 
$\Kr : NC(n) \to NC(n)$ (or as $\Kr_n$, if we need to clarify what 
$n$ we are working with).  The definition of $\Kr_n$ is made by
using partitions of $\{ 1, \ldots , 2n \}$; we take a moment to 
review how this goes, because it illuminates a construction of 
the same kind which we introduce in Section 6.  

Let $\pi$ and $\rho$ be in $NC(n)$.  We will denote by 
$\piodd \sqcup \roeven$ the partition of $\{ 1, \ldots , 2n \}$ 
which is obtained when one turns $\pi$ into a partition of 
$\{ 1,3, \ldots , 2n-1 \}$ and one turns $\rho$ into a partition 
of $\{ 2,4, \ldots , 2n \}$, in the natural way.  That is, 
$\piodd \sqcup \roeven$ has blocks of the form 
$\{ 2i-1 \mid i \in V \}$ where $V$ is a block of $\pi$, and has 
blocks of the form $\{ 2j \mid j \in W \}$ where $W$ is a block of 
$\rho$.  Note that $\piodd \sqcup \roeven$ may not belong to 
$NC(2n)$, due to crossings between its odd and even blocks.  If 
we fix $\pi \in NC(n)$, then it actually turns out that the set
$\{ \rho \in NC(n) \mid \piodd \sqcup \roeven \in NC(2n) \}$
has a largest element with respect to reverse refinement order; 
this largest element is, by definition, the Kreweras complement 
of $\pi$.  That is, $\Kr ( \pi )$ is
defined via the requirement that for $\rho \in NC(n)$ we have:
\[
\piodd \sqcup \roeven \in NC(2n) \ \Leftrightarrow 
\ \rho \leq \Kr ( \pi ). 
\]
Here is a concrete example, considered for $n=5$, which also 
illustrates a standard way of drawing non-crossing partitions.
\\
\begin{center}
  \setlength{\unitlength}{0.3cm}
$\pi$ = 
  \begin{picture}(5,2)
  \thicklines
  \put(1,-1){\line(0,1){2}}
  \put(1,-1){\line(1,0){4}}
  \put(2,0){\line(0,1){1}}
  \put(2,0){\line(1,0){1}}
  \put(3,0){\line(0,1){1}}
  \put(4,-1){\line(0,1){2}}
  \put(5,-1){\line(0,1){2}}
  \put(0.8,1.2){1}
  \put(1.8,1.2){2}
  \put(2.8,1.2){3}
  \put(3.8,1.2){4}
  \put(4.8,1.2){5}
  \end{picture}
$\ $  \hspace{0.2cm} $\Rightarrow \ \piodd \sqcup \kpieven$ =
  \begin{picture}(10,4)
  \thicklines
  \put(1,-2){\line(0,1){4}}
  \put(1,-2){\line(1,0){8}}
  \put(3,0){\line(0,1){2}}
  \put(3,0){\line(1,0){2}}
  \put(5,0){\line(0,1){2}}
  \put(7,-2){\line(0,1){4}}
  \put(9,-2){\line(0,1){4}}
  \linethickness{0.6mm}
  \put(2,-1){\line(0,1){3}}
  \put(2,-1){\line(1,0){4}}
  \put(4, 1){\line(0,1){1}}
  \put(6,-1){\line(0,1){3}}
  \put(8, 1){\line(0,1){1}}
  \put(10,1){\line(0,1){1}}
  \put(0.8,2.2){1}
  \put(1.8,2.2){2}
  \put(2.8,2.2){3}
  \put(3.8,2.2){4}
  \put(4.8,2.2){5}
  \put(5.8,2.2){6}
  \put(6.8,2.2){7}
  \put(7.8,2.2){8}
  \put(8.8,2.2){9}
  \put(9.8,2.2){10}
  \end{picture}
$\ $ \hspace{0.2cm} $\Rightarrow \ \Kr ( \pi )$ =
  \begin{picture}(5,2)
  \thicklines
  \linethickness{0.6mm}
  \put(1,-1){\line(0,1){2}}
  \put(1,-1){\line(1,0){2}}
  \put(2,0){\line(0,1){1}}
  \put(3,-1){\line(0,1){2}}
  \put(4,0){\line(0,1){1}}
  \put(5,0){\line(0,1){1}}
  \put(0.8,1.2){1}
  \put(1.8,1.2){2}
  \put(2.8,1.2){3}
  \put(3.8,1.2){4}
  \put(4.8,1.2){5}
  \end{picture}
\end{center}

$\ $

$\ $

For more details on the $\Kr$ map, see \cite{NS2006}, pp.~147-148.
A neat fact which will be used in Section 5 below is that one
has
\[
| \, \Kr_n ( \pi ) \, | \ = \ n+1 - | \pi |,
\ \ \forall \, \pi \in NC(n),
\]
where $| \pi |$, $| \, \Kr_n ( \pi ) \, |$ denote the numbers
of blocks of the partitions in question.

\vspace{6pt}

$4^o$ The M\"obius function of $NC(n)$ will be denoted 
as $\moeb$ (or as $\moeb_n$, if we need to clarify what $n$ we are 
working with).  This function is defined on 
$\{ ( \pi, \rho ) \mid \pi, \rho \in NC(n)$, 

\noindent
$\pi \leq \rho \}$.  We will actually only use two special cases
of $\moeb$.  The first case is when $\pi = 0_n$; here we simply get
(see \cite{NS2006}, pp.~162-164)
\[
\moeb ( 0_n , \rho )   
= \prod_{W \in \rho} (-1)^{|W|-1} C_{|W|-1},
\]
where for $k \in \bN \cup \{ 0 \}$ we denote 
\[
C_k := \frac{(2k)!}{k! (k+1)!} \ \ \ 
\mbox{(the $k$-th Catalan number).}
\]
The second case we will encounter is the one having $\rho = 1_n$, 
which reduces to the above via the immediate observation that 
$\moeb ( \pi, 1_n ) = \moeb ( 0_n, \Kr ( \pi ) )$.
\end{notation}

$\ $

\begin{remark}   \label{rem:2.1}
In the description of $\Kr$ we used tacitly
the fact that one can talk about the lattice of non-crossing 
partitions $NC(X)$ for any finite totally ordered set $X$
(in particular for $X = \{ 2,4, \ldots , 2n \}$).  The lattice
$NC(X)$ can be of course canonically identified to $NC(n)$ for 
$n = |X|$, upon labelling the elements of $X$ as 
$1,2, \ldots , n$ in increasing order.

Another natural convention used in Notation \ref{def:2.1}.3 was 
that if $X$ and $Y$ are two disjoint finite sets, then we can put 
together a partition $\pi$ of $X$ with a partition $\rho$ of $Y$
in order to obtain a partition denoted as ``$\pi \sqcup \rho$'' 
of $X \cup Y$.  If $X \cup Y$ (hence each of $X,Y$ as well) is 
totally ordered and if we start with $\pi \in NC(X)$ and
$\rho \in NC(Y)$, then it may or may not be that 
$\pi \sqcup \rho \in NC( X \cup Y)$ --- the definition of the 
Kreweras complementation map is actually built around this fact.

We now move to the review of free cumulants.
\end{remark}

$\ $

\begin{notation}    \label{def:2.2}
{\em  [Restrictions of $n$-tuples.] }

\noindent
In order to write more concisely various formulas that will 
appear in the paper, it is convenient to use the following 
natural convention of notation.
Let $\cX$ be a non-empty set, let $n$ be a positive integer,
and let $(x_1, \ldots , x_n)$ be an $n$-tuple in $\cX^n$.  For a 
subset $V = \{ i_1, \ldots , i_m \} \subseteq \{ 1, \ldots , n \}$,
with $1 \leq m \leq n$ and $1 \leq i_1 < \cdots < i_m \leq n$,
we denote
\[
( x_1, \ldots , x_n ) \mid V :=
( x_{i_1}, \ldots , x_{i_m} ) \in \cX^m .
\]

We will use this notation in two ways: one of them (already 
appearing in the next definition) is when $\cX$ is an algebra 
$\cA$ of noncommutative random variables, and the other is when 
$\cX = \{ 1,* \}$ and we talk about the restriction 
$\omega \mid V \in \{ 1,* \}^m$ of a string 
$\omega \in \{ 1, * \}^n$.
\end{notation}

$\ $

\begin{definition}     \label{def:2.3}
{\em [Free cumulant functionals and $R$-transforms.]}

\noindent
Let $( \cA , \varphi )$ be a noncommutative probability space.

\vspace{6pt}

$1^o$ For every $n \in \bN$, the $n$-th {\em moment functional}
of $( \cA, \varphi )$ is the multilinear functional 
$\varphi_n : \cA^n \to \bC$ defined by
$\varphi_n ( a_1, \ldots , a_n) := \varphi (a_1 \cdots a_n),
\ \ a_1, \ldots , a_n \in \cA$.

\vspace{6pt}

$2^o$ For every $n \in \bN$, the $n$-th 
{\em free cumulant functional} of $( \cA, \varphi )$ is the 
multilinear functional $\kappa_n : \cA^n \to \bC$ defined by
\begin{equation}    \label{eqn:23b}
\kappa_n (a_1, \ldots , a_n) = \sum_{\pi \in NC(n)}
\Bigl( \moeb ( \pi, 1_n) \cdot
\prod_{V \in \pi}  \varphi_{|V|} 
\, \bigl( \, (a_1, \ldots , a_n) \mid V\, \bigr) \, \Bigr).
\end{equation}
Equation (\ref{eqn:23b}) is referred to as the 
{\em moment--cumulant formula}. 

\vspace{6pt}

$3^o$ Let $a$ be an element of $\cA$.  The formal power series
$R_a (z) := \sum_{n=1}^{\infty} \kappa_n (a, \ldots, a) \, z^n$
is called the {\em R-transform} of $a$. 
\end{definition}

$\ $

\begin{remark}       \label{rem:2.4}
Let $( \cA , \varphi )$ be a noncommutative probability space,
and consider its free cumulant functionals 
$\kappa_n : \cA^n \to \bC$, as above.

\vspace{6pt}

$1^o$  Let $\cB, \cC \subseteq \cA$ be unital subalgebras
which are freely independent.  The fundamental property of 
the $\kappa_n$'s is that $\kappa_n (a_1, \ldots, a_n) = 0$ 
whenever $n \geq 2$, $a_1, \ldots , a_n \in \cB \cup \cC$,
and there are elements from both $\cB$ and $\cC$ among 
$a_1, \ldots , a_n$.  We also record here a consequence 
of this fact --- a formula (presented in \cite{NS2006} on 
pp.~226-227) which expresses an alternating moment 
$\varphi (b_1 c_1 \cdots b_n c_n)$ in terms of
``free cumulants of the $b$'s and moments of the $c$'s'':
\begin{equation}    \label{eqn:24a}
\varphi (b_1 c_1 \cdots b_n c_n) = \sum_{\pi \in NC(n)}
\prod_{V \in \pi} \kappa_{|V|} \bigl( \, 
( b_1, \ldots , b_n ) \mid V \, ) \, \bigr) \cdot
\prod_{W \in \Kr ( \pi )} \varphi_{|W|} 
\bigl( \, ( c_1, \ldots , c_n ) \mid W \, ) \ \bigr).
\end{equation}

$\ $

$2^o$ We will make essential use of a result of 
Krawczyk and Speicher \cite{KS2000} (presented in \cite{NS2006}
on pp.~178-181), which gives a structured summation 
formula for ``free cumulants with products as entries'', as
follows.
Let $\sigma = \{ J_1, \ldots , J_k \} \in NC(n)$ be a partition 
where every block is an interval:
$J_1 = \{ 1, \ldots , j_1 \}, J_2 = \{ j_1 + 1, \ldots , j_2 \},
\ldots , J_k = \{ j_{k-1} +1, \ldots , j_k \}$
for some $1 \leq j_1 < j_2 < \cdots < j_k = n$.  Then for every
$a_1, \ldots , a_n \in \cA$ one has
\[
\kappa_k \bigl( a_1 \cdots a_{j_1}, \ a_{j_1 +1} \cdots a_{j_2}, 
\, \ldots , \, a_{j_{k-1}+ 1} \cdots a_{j_k} \bigr)
\]
\begin{equation}    \label{eqn:24b}
= \sum_{ \begin{array}{c}
{\scriptstyle \pi \in NC(n) \ such}   \\
{\scriptstyle that \ \pi \vee \sigma = 1_n}
\end{array}  }  \ \prod_{V \in \pi} 
\kappa_{|V|} \bigl( \, (a_1, \ldots , a_n) \mid V \, \bigr).
\end{equation}
In the special case when $\sigma = 1_n$, Equation (\ref{eqn:24b})
becomes a formula expressing the moment $\varphi (a_1 \cdots a_n)$ 
in terms of free cumulants; this special case turns out to be 
equivalent to (\ref{eqn:23b}), and also goes (same as 
(\ref{eqn:23b})) under the name of ``moment--cumulant formula''.

$\ $

$3^o$ We also record some useful properties of free cumulants 
which follow immediately from their definition, by taking into 
account obvious symmetries of the lattices $NC(n)$.

\vspace{6pt} 

(a) Invariance under cyclic permutations of entries:
\[
\kappa_n (a_1, \ldots , a_n) 
= \kappa_n (a_m, \ldots , a_n, a_1, \ldots , a_{m-1}), 
\ \ \forall \, 
1 \leq m \leq n  \mbox{ and } a_1, \ldots , a_n \in \cA.
\]

\vspace{6pt} 

(b) Left-right symmetry: if $\cC \subseteq \cA$ is a commutative
subalgebra, then
\[
\kappa_n (c_1, c_2, \ldots , c_n) 
= \kappa_n (c_n, \ldots , c_2, c_1),
\ \ \forall \, 
n \geq 1 \mbox{ and } c_1, \ldots , c_n \in \cC.
\]

\vspace{6pt} 

(c) Left-right symmetry in $*$-probability framework: 
suppose $( \cA , \varphi )$ is a $*$-probability space,
then one has
$\kappa_n (a_n^{*}, \ldots , a_2^{*}, a_1^{*} ) 
= \overline{\kappa_n (a_1, a_2, \ldots , a_n)}$,
$\forall \, 
n \geq 1 \mbox{ and } a_1, \ldots , a_n \in \cA$.

$\ $

$4^o$ Suppose again that $( \cA , \varphi )$ is a $*$-probability 
space.  Then by using the Cauchy-Schwarz inequality for 
the functional $\varphi$, one immediately sees that every unitary 
$u \in \cA$ has $| \varphi (u) | \leq 1$.  As a consequence, it 
follows that 
\begin{equation}   \label{eqn:25a}
\kappa_n ( u_1, \ldots , u_n ) \leq 16^n, \ \ \forall \, n \geq 1
\mbox{ and $u_1, \ldots , u_n \in \cA$ unitaries.}
\end{equation}
The constant $16$ in (\ref{eqn:25a}) appears upon writing 
cumulants in terms of moments as in Equation (\ref{eqn:23b}), 
then by using estimates on the M\"obius function ---
see the discussion on p.~219 of \cite{NS2006}.

From the bound (\ref{eqn:25a}) it is clear that for every unitary 
$u \in \cA$, the $R$-transform $R_u (z)$ (which was introduced in
Definition \ref{def:2.3}.3 as a formal power series) can also be 
viewed as an analytic function on the disc 
$\{ z \in \bC \mid \, |z| < 1/16 \}$.  
\end{remark}   

$\ $

\begin{remark}  \label{rem:2.6} 
{\em [R-diagonal elements.]}

Let $( \cA , \varphi )$ be a $*$-probability space, and 
suppose that $u,q \in \cA$ are such that $u$ is Haar unitary, 
$q = q^{*}$, and $q$ is free from $\{ u, u^{*} \}$.  The element 
$a := u q \in \cA$ is said to be {\em R-diagonal}.
The motivation for this name (introduced in \cite{NS1997}) is 
that there exists a sequence $(\alpha_k )_{k=1}^{\infty}$, 
called the {\em determining sequence} of $a$, such that for 
$\omega = (\omega (1), \ldots , \omega (n)) \in \{ 1,* \}^n$
one has:
\begin{equation}  \label{eqn:26a}
\kappa_n \bigl( a^{\omega (1)}, \ldots , a^{\omega (n)} \, \bigr) 
= \left\{  \begin{array}{cl}
\alpha_{n/2},
    &  \mbox{if $n$ even and } \omega = (1,*,1,*, \ldots ,1,*)   \\
    &  \mbox{ $\ $ or } \omega = (*,1,*,1, \ldots ,*,1),         \\
0,  &    \mbox{otherwise.}
\end{array}  \right.
\end{equation} 
The $\alpha_k$'s can be written in terms of the free cumulants of 
$q^2$ via a formula which looks similar to Equation (\ref{eqn:23b}):
\begin{equation}    \label{eqn:26b}
\alpha_k = \sum_{\pi \in NC(k)}
\Bigl( \moeb ( \pi, 1_k) \cdot
\prod_{V \in \pi}  \kappa_{|V|} (q^2, \ldots , q^2) \, \Bigr),
\ \ k \in \bN .
\end{equation}
The derivation of these facts is presented in \cite{NS2006} on 
pp.~241-244 of Lecture 15.

Note that Equation (\ref{eqn:1e}) from the Introduction corresponds 
to the special case $q= \unitA$ of the above formulas.  Indeed, in 
this case the sum on the right-hand side of (\ref{eqn:26b}) has only 
one non-vanishing term, corresponding to $\pi = 0_k$, and we get
$\alpha_k = \moeb ( 0_k, 1_k) = (-1)^{k-1} C_{k-1}$, as 
stated in (\ref{eqn:1e}).
\end{remark}

$\ $

$\ $

\section{The polynomials $Z_{\omega}$}

\setcounter{equation}{0}

\begin{prop-and-notation}   \label{prop:3.1}
Let $\omega = ( \omega (1), \ldots , \omega (n) )$ be a string
in $\{1, *\}^n$, for some $n \geq 1$.  There exists a polynomial 
$Z_{\omega} \in \bQ [x, y]$, uniquely determined, such that
\begin{equation}   \label{eqn:31a}
\kappa_n \bigl( \, u_t^{ \omega (1) }, \ldots , 
u_t^{ \omega (n) } \, \bigr) = Z_{\omega} (t, e^{-t/2} ),
\ \ \forall \, t \in [ 0, \infty ).
\end{equation}
Moreover, the polynomial $Z_{\omega} (x,y)$ has the form 
\begin{equation}   \label{eqn:31b}
Z_{\omega} (x,y) = \sum_{0 \leq j \leq n/2}
Z_{\omega}^{(n-2j)} (x) \cdot y^{n-2j}, 
\end{equation}
where $Z_{\omega}^{(n-2j)} \in \bQ [x]$ for 
$0 \leq j \leq n/2$.
\end{prop-and-notation}

\begin{proof}  We will give an explicit formula for the polynomial
$Z_{\omega}$.  In order to state it, we introduce some preliminary 
items of notation.  We first recall that Lemma 1 on page 4 of 
\cite{B1997} says that the moments of $u_t$ are
$\varphi_t ( u_t^n ) = Q_n (t) e^{-nt/2}$, $n \geq 1$, where 
\[
Q_n (t) = -\sum_{j=0}^{n-1} 
\frac{(-n)^{j-1}}{j!}\binom{n}{j+1} t^j.
\]
For a string $\omega$ in $\{ 1, * \}^n$ which has $k$ occurrences 
of the symbol ``$1$''  and $\ell = n-k$ occurrences of ``$*$'', 
we then introduce a polynomial $M_{\omega} \in \bQ [x,y]$ defined by
\begin{equation}   \label{eqn:31c}
M_{\omega}(x,y) := \left\{  \begin{array}{ll}
Q_{ |k- \ell | } (x) \, y^{|k- \ell|},  
                            &  \mbox{if $k \neq \ell$}   \\
1,                          &  \mbox{if $k= \ell$.}
\end{array}  \right.
\end{equation}
[For instance, if $n=7$ and $\omega = (1,1,*,1,1,*,1)$ then 
$M_{\omega} (x,y) = Q_3 (x) \, y^3 = (1- 3x + \frac{3}{2}x^2) y^3$.]

Based on (\ref{eqn:31c}), we define the $Z_{\omega}$'s as 
follows: for every $\omega \in \{ 1, * \}^n$ we put
\begin{equation}    \label{eqn:31d}  
Z_{\omega} := \sum_{\pi \in NC(n)} \moeb ( \pi , 1_n ) \cdot
\Bigl( \, \prod_{V \in \pi} M_{\omega \mid V} \, \Bigr) ,
\end{equation}
where the notations related to $NC(n)$ and its M\"obius function
are as in Section 2.  

\noindent
[A concrete example: if $n=3$ and $\omega = (1,*,1)$, then
\[
Z_{(1,*,1)} := M_{(1,*,1)} - M_{(1)} M_{(*,1)} 
- M_{(1,1)} M_{(*)} - M_{(1,*)} M_{(1)} 
+ 2 M_{(1)} M_{(*)} M_{(1)} ;
\]
this comes, upon substituting the $M$'s, to 
$Z_{(1,*,1)} (x,y) = (x+1) y^3 - y$.]

Fix a $t \in [0, \infty)$, and for every $n \in \bN$ let
$\varphi_n : \cA_t^n \to \bC$ be the $n$-th moment functional 
of $( \cA_t, \varphi_t )$.  Then one has
\begin{equation}   \label{eqn:31e}
\varphi_n \bigl( \, u_t^{\omega (1)}, \ldots , 
u_t^{\omega (n)} \, \bigr) = M_{\omega} (t, e^{-t/2}), 
\ \ \forall \, \mbox{$n \in \bN$ and $\omega \in \{ 1,* \}^n$.}
\end{equation}
This in turn implies that, for every $n \in \bN$ and 
$\omega \in \{ 1,* \}^n$:
\[
\kappa_n 
\bigl( \, u_t^{\omega (1)}, \ldots , u_t^{\omega (n)} \, \bigr)
= \sum_{\pi \in NC(n)} \moeb ( \pi , 1_n ) \cdot 
\Bigl( \, \prod_{V \in \pi} 
\varphi_{|V|}  \bigl( \, ( u_t^{\omega (1)}, \ldots , 
                           u_t^{\omega (n)} ) \mid V  \, \bigr) 
\, \Bigr) 
\]
\[
= \sum_{\pi \in NC(n)} \moeb ( \pi , 1_n ) \cdot 
\bigl( \, \prod_{V \in \pi} 
M_{\omega \mid V} (t, e^{-t/2}) \, \bigr) 
= Z_{\omega} (t, e^{-t/2})
\]
(where we first used the moment--cumulant formula (\ref{eqn:23b}), 
then we invoked Equations (\ref{eqn:31e}) and (\ref{eqn:31d})).  
Thus $Z_{\omega}$ has the property stated in 
Equation (\ref{eqn:31a}).

The uniqueness of $Z_{\omega}$ with the property stated in 
Equation (\ref{eqn:31a}) follows from general considerations 
(a polynomial in $\bQ [x,y]$ is determined by its values
on pairs $(t, e^{-t/2})$, with $t \in [ 0, \infty )$).

Finally, let us also verify the specific form of $Z_{\omega}$ that
was indicated in Equation (\ref{eqn:31b}).  It suffices to show 
that: for every $\pi \in NC(n)$, 
the term indexed by $\pi$ in the sum on the right-hand side of 
(\ref{eqn:31d}) is of the form $y^s \cdot T(x)$, where 
$s \in \{ 0, 1 , \ldots , n \}$ has same parity as $n$, and 
where $T \in \bQ [x]$.  So fix a partition 
$\pi = \{ V_1, \ldots , V_k \} \in NC(n)$, and for every 
$1 \leq j \leq k$ denote by $p_j$ and by $q_j$ the number of 
occurrences of ``$1$'' and respectively ``$*$'' in the restricted 
word $\omega | V_j$. The term indexed by $\pi$ in the sum on 
the right-hand side of (\ref{eqn:31d}) is
$\moeb ( \pi , 1_n) \cdot 
\prod_{j=1}^k Q_{|p_j-q_j|} (x) y^{|p_j-q_j|},$
where we set $Q_0 :=1$. This is indeed of the form $y^s \cdot T(x)$, 
with $s := \sum_{j=1}^k |p_j - q_j|$, and we are only left to check 
that $n-s$ is an even non-negative integer. But the latter fact 
follows from 
\[
n-s = \sum_{j=1}^k (p_j + q_j) - \sum_{j=1}^k |p_j - q_j|
= \sum_{j=1}^k (p_j + q_j - |p_j - q_j| ),
\]
where every $p_j + q_j - |p_j - q_j|$ is an even non-negative 
integer.
\end{proof}

$\ $

\begin{example}
Here are some concrete examples of polynomials $Z_{\omega}$:
\[
\begin{array}{lcl}
Z_{(1,*)} (x,y)      & = & - y^2 + 1,                           \\ 
Z_{(1,1,*)} (x,y)    & = & (x+1) y^3 - y,                       \\
Z_{(1,1,1,*)} (x,y)  & = &
                - ( \frac{3}{2} x^2 + 2x + 1) y^4 + (x+1) y^2,  \\
Z_{(1,1,*,*)} (x,y)  & = &  - (x^2 + 2x + 2) y^4 + 2 y^2,       \\
Z_{(1,*,1,*)} (x,y)  & = & - (2x+3) y^4 + 4 y^2 - 1.            \\
\end{array}
\]
If we add to this list the formula for a polynomial 
$Z_{(1,1, \ldots , 1)}$ from Equation (\ref{eqn:1gg}), and if 
we take into account some obvious invariance properties of the 
$Z_{\omega}$'s (as recorded in the next remark), then  
these examples cover all strings 
$\omega \in \{ 1 , * \}^n$ for $n \leq 4$.
\end{example}

$\ $

\begin{remark}  \label{rem:3.2}
The polynomials $Z_{\omega}$ have some invariance properties 
which follow directly from their definition.

\vspace{6pt}

$1^o$ Let $\omega, \omega ' \in \{ 1, * \}^n$ be such that
$\omega '$ is obtained from $\omega$ via a cyclic permutation.
The invariance of free cumulants under cyclic permutations 
of entries gives 
$Z_{\omega} (t, e^{-t/2}) = Z_{\omega '} (t, e^{-t/2})$,
$t \in [0 , \infty )$, which implies that the polynomials 
$Z_{\omega}$ and $Z_{\omega '}$ coincide.

\vspace{6pt}

$2^o$ The same conclusion as in $1^o$ holds if we take $\omega '$
to be obtained from $\omega$ by reversing the order of its 
components, $\omega ' = ( \omega (n), \ldots , \omega (1) )$.
(This time we use the invariance property of free cumulants 
that was reviewed in Remark \ref{rem:2.4}.3(b).) 

\vspace{6pt}

$3^o$ The moments of the free unitary Brownian motion $u_t$
are real numbers (as reviewed at the beginning of the preceding
proof).  This has the consequence that $u_t^{*}$ can also serve 
as free unitary Brownian motion at time $t$, which in turn 
implies that
\[
\kappa_n \bigl( \, u_t^{\omega (1)}, \ldots , 
                   u_t^{\omega (n)} \, \bigr) =
\kappa_n \bigl( \, u_t^{\omega' (1)}, \ldots , 
                   u_t^{\omega' (n)} \, \bigr),
\ \ \forall \, t \in [ 0, \infty ),
\]
with $\omega '$ obtained out of $\omega \in \{ 1,* \}^n$ 
by swapping the roles of $1$ and $*$ (every $\omega (i)$ which 
is a $1$ is replaced by a $*$, and vice-versa).  The uniqueness 
property of $Z_{\omega}$ thus implies that 
$Z_{\omega} = Z_{\omega '}$ in this situation, too.
\end{remark}

$\ $

We next put into evidence a very useful recursion satisfied by 
the polynomials $Z_{\omega}$.  This is done in Proposition 
\ref{prop:3.4}; the essence of 
the argument is a calculation which holds for any unitary in 
a $*$-probability space, and is presented in the next lemma.

$\ $

\begin{lemma}   \label{lemma:3.3}
Let $( \cA , \varphi )$ be a $*$-probability space, and let 
$u \in \cA$ be a unitary element.  Consider a string
$\omega = ( \omega (1), \ldots , \omega (n)) \in \{ 1,* \}^n$
with $n \geq 3$ and where $\omega (1) = 1$, $\omega (n) = *$.  
Then
\begin{equation}  \label{eqn:33a}
\kappa_n 
\bigl( \, u^{\omega (1)}, \ldots , u^{\omega (n)} \, \bigr) 
= - \sum_{m=1}^{n-1}  \kappa_m 
\bigl( \, u^{\omega (1)}, \ldots , u^{\omega (m)} \, \bigr)
\cdot  \kappa_{n-m}
\bigl( \, u^{\omega (m+1)}, \ldots , u^{\omega (n)} \, \bigr) 
\end{equation}
(where $\kappa_n, \kappa_m, \kappa_{n-m}$ denote free cumulant 
functionals for $( \cA , \varphi )$).
\end{lemma}

\begin{proof}  
We may assume (by replacing $\cA$ with the $*$-algebra generated 
by $u$) that $( \cA , \varphi )$ is tracial.  In particular, we 
can write
\[
\kappa_n 
\bigl( \, u^{\omega (1)}, \ldots , u^{\omega (n)} \, \bigr) 
= \kappa_n
\bigl( \, u^{\omega (n)}, u^{\omega (1)}, \ldots , 
u^{\omega (n-1)} \, \bigr) 
= \kappa_n
\bigl( \, u^{*}, u, u^{\omega (2)}, \ldots , 
u^{\omega (n-1)} \, \bigr) .
\]

Now, we know that 
$\kappa_{n-1} \bigl( \, u^{*} u, u^{\omega (2)}, \ldots , 
u^{\omega (n-1)} \, \bigr) = 0$
(a free cumulant of length $\geq 2$ always vanishes when one of 
its entries is equal to $1_{ { }_{\cA} }$).  On the other hand, 
the formula (\ref{eqn:24b}) for cumulants with products as 
entries gives
\begin{equation}  \label{eqn:33b}
\kappa_{n-1}
\bigl( \, u^{*} u, u^{\omega (2)}, \ldots , 
u^{\omega (n-1)} \, \bigr) 
\end{equation}
\[
= \sum_{ \begin{array}{c}
{\scriptstyle \pi \in NC(n) \ such}  \\
{\scriptstyle that \ \pi \vee \sigma = 1_{n}}
\end{array}  } \ \prod_{V \in \pi}  \kappa_{|V|} 
\bigl( \, (u^{*}, u, u^{\omega (2)}, \ldots , 
u^{\omega (n-1)}) \mid V \, \bigr),
\]
where $\sigma \in NC(n)$ is the partition consisting of the 
2-element block $\{ 1, 2 \}$ and of $n-2$ blocks with one element.

Let $\pi \in NC(n)$ be such that $\pi \vee \sigma = 1_{n}$, and 
let $V'$ and $V''$ be the blocks of $\pi$ which contain $1$ and $2$, 
respectively.  We observe that $V' \cup V'' = \{ 1, \ldots ,n \}$; 
indeed, in the opposite case we could consider the partition 
$\widetilde{\pi} \in NC(n)$ which is obtained from $\pi$ by 
joining together the blocks $V$ and $V'$, and\footnote{ The partition $\widetilde{\pi}$ thus consists of 
$V' \cup V''$ and of all the blocks $V \in \pi$ such that 
$V \neq V', V''$.  The detail which prevents $\widetilde{\pi}$ 
from having crossings is that $V'$ and $V''$ contain the adjacent 
points $1$ and $2$.}
this $\widetilde{\pi}$ would satisfy 
$\pi, \sigma \leq \widetilde{\pi} \neq 1_{n}$, in contradiction
with the assumption that $\pi \vee \sigma = 1_{n}$.  If $V' = V''$ 
then $\pi = 1_n$.  If $V' \neq V''$ then either $V' = \{ 1 \}$, 
$V'' = \{ 2,3, \ldots ,n \}$, or $V''$ is nested inside $V'$.  In the 
latter case, denoting $| V'' | =: m$, we find that 
$V'' = \{ 2, \ldots , m+1 \}$ and 
$V' = \{ 1 \} \cup \{ m+2, \ldots , n \}$, where $1 \leq m \leq n-2$.

The conclusion of the discussion in the preceding paragraph is that
the sum on the right-hand side of (\ref{eqn:33b}) can be written 
explicitly as 
\begin{equation}  \label{eqn:33c}
\kappa_n
\bigl( \, u^{*}, u, u^{\omega (2)}, \ldots , 
u^{\omega (n-1)} \, \bigr)  
+ \sum_{m=1}^n \ \kappa_{\pi_m} 
\bigl( \, u^{*}, u, u^{\omega (2)}, \ldots , 
u^{\omega (n-1)} \, \bigr),
\end{equation}  
with $\pi_m = \{ \, \{ 2, \ldots , m+1 \} , \, 
\{ 1, \ldots ,n \} \setminus \{ 2, \ldots , m+1 \} \, \}$,
$1 \leq m \leq n-1$.  It is immediately seen (by doing the 
suitable cyclic permutation of entries in
$\kappa_{n-m} \bigl( \, u^{*}, u^{\omega (m+1)}, \ldots , 
u^{\omega (n-1)} \, \bigr)$) that for every 
$1 \leq m \leq n-1$ one has
\[
\kappa_{\pi_m} 
\bigl( \, u^{*}, u, u^{\omega (2)}, \ldots , 
u^{\omega (n-1)} \, \bigr)
= \kappa_m 
\bigl( \, u^{\omega (1)}, \ldots , u^{\omega (m)} \, \bigr)
\cdot  \kappa_{n-m}
\bigl( \, u^{\omega (m+1)}, \ldots , u^{\omega (n)} \, \bigr).
\]
But the sum in (\ref{eqn:33c}) is equal to $0$ (since it started 
as an expansion for $\kappa_{n-1} (u^{*}u, \ldots ) = 0$), and 
the statement of the lemma follows.
\end{proof}

$\ $

\begin{proposition}   \label{prop:3.4}
Suppose that $n \geq 3$ and that 
$\omega = ( \omega (1), \ldots , \omega (n)) \in \{ 1,* \}^n$
has $\omega (1) = 1$ and $\omega (n) = *$.  Then it follows that 
\begin{equation}  \label{eqn:34a}
Z_{\omega} = - \sum_{m=1}^{n-1}
Z_{ ( \omega (1), \ldots , \omega (m) )} \cdot
Z_{ ( \omega (m+1), \ldots , \omega (n) )}
\end{equation}
(equality of polynomials in two variables).
\end{proposition}

\begin{proof}
Let $Z \in \bQ [x,y]$ be the polynomial which appears on the 
right-hand side of (\ref{eqn:34a}).  By using 
Lemma \ref{lemma:3.3} one immediately sees that 
$Z(t, e^{-t/2}) = \kappa_n 
\bigl( \, u_t^{\omega (1)}, \ldots , u_t^{\omega (n)} \, \bigr),
\ \ t \in [0, \infty )$,
and this implies $Z = Z_{\omega}$.
\end{proof}

$\ $

As an application of Proposition \ref{prop:3.4} we show the 
following: the number of relevant terms (counting from the top) 
in the expansion given for $Z_{\omega}$ in Equation (\ref{eqn:31b})
is limited by how many times we 
switch between the symbols `$1$' and `$*$', upon going around the
string $\omega$.  For instance if $\omega = (1,1, \ldots , 1)$
then the expansion (\ref{eqn:31b}) amounts to just 
$Z_{\omega} (x,y) = Z_{\omega}^{(n)} (x) \cdot y^n$,
if $\omega = (1,\ldots , 1,*, \ldots , *)$
then $Z_{\omega} (x,y) = Z_{\omega}^{(n)} (x) \cdot y^n
+ Z_{\omega}^{(n-2)} (x) \cdot y^{n-2}$, and so on.  This fact is
stated precisely in Theorem \ref{thm:3.7} below.  In order to come
to it, we first record the (natural) definition for what is the 
``number of switches between $1$ and $*$'' in a string $\omega$.

$\ $

\begin{def-and-remark}  \label{def:3.5}
For every $n \in \bN$ and $\omega \in \{ 1, * \}^n$ we 
define the {\em switch-number} of $\omega$ to be
\begin{equation}  \label{eqn:35a}
\switch ( \omega ) 
:= \deltabar_{\omega (n), \omega (1)} 
+  \deltabar_{\omega (1), \omega (2)} + \cdots 
+  \deltabar_{\omega (n-1), \omega (n)},
\end{equation}
where the $\deltabar$'s on the right-hand side of the 
equation are assigned by putting
\[
\deltabar_{1,*} = \deltabar_{*,1} = 1 \mbox{ and }
\deltabar_{1,1} = \deltabar_{*,*} = 0.
\]
It is easily seen that 
$\switch ( \omega )$ is an even integer such that 
$0 \leq \switch ( \omega ) \leq n$.  Another immediate 
observation is that 
$\switch ( \omega ) = \switch ( \omega ' )$ whenever
$\omega '$ is obtained out of $\omega$ via one of the 
transformations discussed in Remark \ref{rem:3.2}.
\end{def-and-remark}

$\ $

\begin{lemma}   \label{lemma:3.6}
Let $n \geq 3$ be an integer, and let 
$\omega = ( \omega (1), \ldots , \omega (n)) \in \{ 1,* \}^n$
be such that $\omega (1) = 1$, $\omega (n) = *$.  Let 
$m$ be a number in $\{ 1, \ldots, n-1 \}$, and consider the 
strings 
\[
\omega ' := 
( \omega (1), \ldots , \omega (m)) \in \{ 1,* \}^m, \ \ 
\omega '' := 
( \omega (m+1), \ldots , \omega (n)) \in \{ 1,* \}^{n-m}.
\]
Then we have
\begin{equation}  \label{eqn:36a}
\switch ( \omega ' ) + \switch ( \omega '' ) 
\leq  \switch ( \omega ). 
\end{equation}
\end{lemma}

\begin{proof} We will prove the required inequality by 
assuming that $1 < m < n-1$ (the cases when $m=1$ or $m=n-1$
are analogous, and simpler).  Look at the difference
\[
\switch ( \omega ) - \Bigl( \,
\switch ( \omega ' ) + \switch ( \omega '' )  \, \Bigr).
\]
By cancelling common terms in the expressions which define
these three switch-numbers, we find that the above difference 
is equal to 
\[
\bigl( \, \deltabar_{\omega (m), \omega (m+1)} 
+ \deltabar_{\omega (n), \omega (1)} \, \bigr) \, - \,
\bigl( \, \deltabar_{\omega (m), \omega (1)} 
+  \deltabar_{\omega (n), \omega (m+1)} \, \bigr).
\]
Since $\deltabar_{\omega (n), \omega (1)} =1$ (while the other 
$\deltabar$'s appearing above are $0$ or $1$), we get that
\[
\switch ( \omega ) - \Bigl( \,
\switch ( \omega ' ) + \switch ( \omega '' )  \, \Bigr) \geq -1.
\]
But switch-numbers are always even; so in the latter 
inequality we are actually forced to have ``$\geq 0$'', and 
(\ref{eqn:36a}) follows.
\end{proof}

$\ $

\begin{theorem}  \label{thm:3.7}
Let $n$ be a positive integer, and let $\omega$ be a string in 
$\{ 1,* \}^n$.  Consider the polynomial $Z_{\omega} (x,y)$ and 
its expansion as sum of terms 
$Z_{\omega}^{(n-2j)} (x) \cdot y^{n-2j}$, with $0 \leq j \leq n/2$,
which was obtained in Proposition \ref{prop:3.1}.  One has
\begin{equation}  \label{eqn:37a}
Z_{\omega}^{(n-2j)} = 0 \ \mbox{ whenever } 
2j > \switch ( \omega ).
\end{equation}
\end{theorem}

\begin{proof}  We first observe that the statement 
of the theorem holds when $\omega$ is of the form 
$(1,1, \ldots , 1)$ or $(*,*, \ldots , *)$.  In this case 
we have $\switch ( \omega ) = 0$; so Equation (\ref{eqn:37a})
says that $Z_{\omega}^{(n-2j)} = 0$ for every $j \neq 0$,
i.e.~that 
$Z_{\omega} (x,y) = Z_{\omega}^{(n)} (x) \cdot y^n$.
This is indeed true, as noticed in Equation (\ref{eqn:1gg}) of 
the introduction.

We now prove by induction on $n$ 
that the statement of the theorem holds for general 
strings $\omega \in \{ 1,* \}^n$.  The case $n=1$
is included in the preceding paragraph.  Let us also verify the 
case $n=2$.  In this case, the strings $(1,1)$ and $(*,*)$
are covered by the preceding paragraph, while the strings 
$(1,*)$ and $(*,1)$ have switch-number equal to $2$ --- so for 
the latter two strings, Equation (\ref{eqn:37a}) is fulfilled 
vacuously (there is no $j$ in the range $0 \leq j \leq n/2$ such 
that $2j > \switch ( \omega )$).

In the remaining part of the proof we do the induction step:
we fix an integer $n \geq 3$, we assume that the statement of the
theorem holds for strings of length $\leq n-1$, and we will prove 
that it also holds for strings of length $n$. 

So let us also fix an
$\omega = ( \omega (1), \ldots , \omega (n))$ 
in $\{ 1,* \}^n$, for which we will verify that (\ref{eqn:37a}) 
holds.  We distinguish three cases.

\vspace{6pt}

{\em Case 1.} $\omega = (1,1, \ldots, 1)$ or 
$\omega = (*,*, \ldots ,*)$.

This case was verified in the first paragraph of the proof.

\vspace{6pt}

{\em Case 2.} $\omega$ is such that $\omega (1) = 1$ and 
$\omega (n) = *$.  

Consider a $j \in \bN$ such that $0 \leq j \leq n/2$ and such 
that $2j > \switch ( \omega )$.  (We assume that such $j$'s exist,
otherwise there is nothing to prove.)  We are in a situation 
where we can invoke Proposition \ref{prop:3.4}.  By extracting 
the coefficient of $y^{n-2j}$ on both sides of the recursion 
provided by that proposition, we find that
\begin{equation}  \label{eqn:37b}
Z_{\omega}^{(n-2j)} = - \sum_{m=1}^{n-1} \
\sum_{ \begin{array}{c}
{\scriptstyle 0 \leq k \leq m/2, \ 0 \leq \ell \leq (n-m)/2}  \\
{\scriptstyle such \ that \ k + \ell = j}
\end{array} } \ \
Z_{ ( \omega (1), \ldots , \omega (m) ) }^{(m-2k)} \cdot
Z_{ ( \omega (m+1), \ldots , \omega (n) ) }^{(n-m-2 \ell)} 
\end{equation}
(equality of polynomials in $\bQ [x]$).
We will show that $Z_{\omega}^{(n-2j)} = 0$ by verifying
that every term in the double sum on the right-hand side of 
(\ref{eqn:37b}) is the zero polynomial.  Indeed, let us 
pick such a term (indexed by an $m$, and then by a pair 
$(k, \ell )$), and let us denote
\[
\omega ' := 
( \omega (1), \ldots , \omega (m)) \in \{ 1,* \}^m, \ \ 
\omega '' := 
( \omega (m+1), \ldots , \omega (n)) \in \{ 1,* \}^{n-m}.
\]
We have 
$2k + 2 \ell = 2j > \switch ( \omega ) 
\geq \switch ( \omega ' ) + \switch ( \omega '')$
(where at the second inequality we use Lemma \ref{lemma:3.6}).
So either $2k > \switch ( \omega ')$ or
$2 \ell > \switch ( \omega '')$, and the induction hypothesis
gives us that either
$Z_{ ( \omega (1), \ldots , \omega (m) ) }^{(m-2k)} = 0$ or
$Z_{ ( \omega (m+1), \ldots , \omega (n) ) }^{(n+m-2 \ell)} = 0$.
Either way, the product of the latter two polynomials is 
$0$, and this completes the verification of Case 2.

\vspace{6pt}

{\em Case 3.} $\omega$ does not fall in either Case 1 
or Case 2 above.

Since $\omega$ is not in Case 1, both symbols $1$ and $*$
must appear among its components.  It is then easy to see that 
there exists a string $\omega '$ obtained from $\omega$ via a 
cyclic permutation of components, such that $\omega ' (1) = 1$ 
and $\omega ' (n) = *$.  The string $\omega '$ has 
$Z_{\omega '} = Z_{\omega}$ (Remark \ref{rem:3.2}.1), and has
$\switch ( \omega ' ) = \switch ( \omega )$ 
(Remark \ref{def:3.5}).  For any $j$ such that 
$2j > \switch ( \omega ) = \switch ( \omega ' )$ we have 
$Z_{\omega'}^{(n-2j)} = 0$, because $\omega '$ falls in the 
Case 2 discussed above.  It follows that 
$Z_{\omega}^{(n-2j)} = 0$ as well.  This concludes the 
verification of the induction step, and the proof of the theorem.
\end{proof}

$\ $

\section{A special case of $Z_{\omega}$'s}

\setcounter{equation}{0}

In the present section we determine what is the polynomial 
$Z_{\omega}$ for a string of the form 
$\omega = (1, \ldots , 1, *, \ldots , *)$, having $k$ 
occurrences of ``$1$'' followed by $\ell$ occurrences of ``$*$'', 
for some $k, \ell \geq 1$.  In this case, Theorem \ref{thm:3.7}
says that the expansion from 
Equation (\ref{eqn:1g}) is reduced to its top two terms:
\[
Z_{\omega} (x,y) = Z_{\omega}^{(k+ \ell) } (x) \, y^{k + \ell} 
+ Z_{\omega}^{(k + \ell -2)} (x) \, y^{k + \ell -2}.
\]
We will retrieve this fact, and we will moreover show how the 
polynomials $Z_{\omega}^{(k+ \ell) } (x)$,
$Z_{\omega}^{(k + \ell -2)} (x)$ can be written 
explicitly as some Laplace transform integrals.  We start with 
a calculation (consequence of the above Lemma \ref{lemma:3.3})
which holds for any unitary in a $*$-probability space. 

$\ $

\begin{lemma}   \label{lemma:4.1}
Let $( \cA , \varphi )$ be a $*$-probability space and let 
$u \in \cA$ be a unitary element.  It makes sense to define 
an analytic function 
$F_u : \{ (z,w) \in \bC^2 \mid \, |z|, |w| < 1/16 \} \to \bC$ 
by putting
\begin{equation}  \label{eqn:41a}
F_u (z,w) := \sum_{k, \ell =1}^{\infty} 
\kappa_{k+ \ell } ( \,  \underbrace{u, \ldots , u}_k 
\, , \, \underbrace{u^{*}, \ldots , u^{*} }_{\ell} \, ) \,
z^k w^{\ell}.
\end{equation}
Moreover, there exists $r \in (0,  1/16)$ such that for 
$|z|, |w| < r$ one has
\begin{equation}  \label{eqn:41b}
F_u (z,w) =  \frac{zw - R_u (z)  R_{u^{*}} (w)}{ 1
+ R_u (z) + R_{u^{*}} (w) } ,
\end{equation}
where $R_u, R_{u^{*}} : \{ z \in \bC \mid \, |z| < 1/16 \} \to \bC$ 
are $R$-transforms (as discussed in Remark \ref{rem:2.4}.4).  
\end{lemma}

\begin{proof}
The fact that $F_u (z,w)$ defined in Equation (\ref{eqn:41a}) is 
well-defined and analytic on 
$\{ (z,w) \in \bC^2 \mid \, |z|, |w| < 1/16 \}$ follows from the 
bound 
\[
\kappa_{k+ \ell } ( \,  \underbrace{u, \ldots , u}_k 
\, , \, \underbrace{u^{*}, \ldots , u^{*} }_{\ell} \, )
\leq 16^{k + \ell}, \ \ k, \ell \in \bN ,
\]
which was mentioned in Remark \ref{rem:2.4}.4.

Let us next consider the analytic function defined on 
$\{ (z,w) \in \bC^2 \mid \, |z|, |w| < 1/16 \}$ by 
\begin{equation}   \label{eqn:41c}
(z,w) \mapsto F_u (z,w) \, ( 1 + R_u (z) + R_{u^{*}} (w) ) 
+ R_u (z)  R_{u^{*}} (w) .
\end{equation}
We will prove that this function is just $(z,w) \mapsto zw$;
the formula claimed in the lemma for $F(z,w)$ will then clearly 
follow (with $r$ picked e.g.~such that 
$| R_u (z) | < 1/2$ for $|z| < r$).

It thus suffices to prove that the coefficient of $z^k w^{\ell}$ 
in the analytic function from (\ref{eqn:41c}) is equal to $1$ for 
$k = \ell =1$, and is equal to $0$ for any $(k, \ell ) \neq (1,1)$ 
in $\bN^2$.  In the special case when $k = \ell =1$, the 
coefficient in question comes out as 
$\kappa_2 ( u, u^{*} ) + \kappa_1 (u) \, \kappa_1 ( u^{*} )$,
which is equal to $\varphi ( uu^{*} )$ by the moment-cumulant 
formula, and thus is indeed equal to $1$.  The cases when 
$(k, \ell ) \neq (1,1)$ are covered by Lemma \ref{lemma:3.3}.
Indeed, let us say for instance that both $k$ and $\ell$ are 
$\geq 2$ (if $k = 1 < \ell$ or if $\ell = 1 < k$ then the argument 
is analogous, and shorter).  Direct inspection shows that the 
coefficient of $z^k w^{\ell}$ in the function from (\ref{eqn:41c}) 
is equal to
\[
\kappa_{k+ \ell } ( \,  \underbrace{u, \ldots , u}_k \, , 
\,  \underbrace{u^{*}, \ldots , u^{*}}_{\ell} \, ) 
+ \sum_{i=1}^{k-1} \kappa_i (u, \ldots , u) \cdot
\kappa_{(k-i)+ \ell} ( \,  \underbrace{u, \ldots , u}_{k-i} \, , 
\,  \underbrace{u^{*}, \ldots , u^{*}}_{\ell} \, )
\]
\[
+ \sum_{j=1}^{\ell -1} 
\kappa_{k+ ( \ell -j)} ( \,  \underbrace{u, \ldots , u}_k \, , 
\,  \underbrace{u^{*}, \ldots , u^{*}}_{\ell -j} \, ) \cdot
\kappa_j (u^{*}, \ldots , u^{*})
+ \kappa_k (u, \ldots , u) \cdot 
  \kappa_{\ell} (u^{*}, \ldots , u^{*}).
\]
But Lemma \ref{lemma:3.3} (used for the string in 
$\{ 1,* \}^{k + \ell}$ which has $k$ occurrences of $1$ followed 
by $\ell$ occurrences of $*$) says precisely that the latter sum 
is equal to $0$.
\end{proof}

$\ $

We now turn to the case of interest, of the free unitary 
Brownian motion.

$\ $

\begin{lemma}  \label{lemma:4.2}
Let us fix $t \in ( 0, \infty )$.  In the framework of 
Lemma \ref{lemma:4.1} let us put $u = u_t$ (free unitary
Brownian motion at time $t$), and let us consider the 
analytic function $F_{u_t} (z,w)$ defined as in Equation
(\ref{eqn:41a}).  Then for $|z|, |w|$ small enough we have
that
\begin{equation}  \label{eqn:42a}
F_{u_t} (z,w) = \frac{1}{t} \cdot \frac{ t^2zw - 
W( te^{-t/2}z ) W( te^{-t/2}w ) }{ t+ 
W( te^{-t/2}z ) + W( te^{-t/2}w ) },
\end{equation}
or equivalently that
\begin{equation}  \label{eqn:42b}
F_{u_t} (z,w) = tzw  \, \int_0^1 e^{-ts} \,
e^{-sW( te^{-t/2} z )} \, 
e^{-sW( te^{-t/2} w )} \, ds,
\end{equation}
with $W$ the Lambert function (viewed here as analytic 
function on $\{ z \in \bC \mid \, |z| < 1/e \}$).
\end{lemma}

\begin{proof}
As mentioned in the introduction, the rescaling $v_t = e^{t/2}u_t$ 
has $R$-transform $R_{v_t} (z) = t^{-1} W(tz)$.  But
$R_{u_t} (z) = R_{v_t} ( e^{-t/2}z )$, so we find the 
$R$-transform of $u_t$ to be
\begin{equation}  \label{eqn:42c}
R_{u_t} (z) =  \frac{1}{t} W(t e^{-t/2} z).
\end{equation}
The equality (\ref{eqn:42c}) holds when $|z|$ is small enough so 
that both sides are defined (one can e.g.~use $|z| < 1/16$ on 
the left-hand side and $|z| < e^{t/2}/(et)$ on the right-hand 
side).  The 
adjoint $u_t^{*}$ has the same $R$-transform as $u_t$ itself.
We replace all this into the result of Lemma \ref{lemma:4.1}.
Upon also requiring the condition that $|z|, |w|$ are small enough 
such that
\[
| W( t e^{-t/2} z)|, \, | W( t e^{-t/2} w)| < t/2
\]
(which ensures that the denominator 
$t + W( t e^{-t/2} z) + W( t e^{-t/2} w)$ does not vanish),
we arrive to the formula for $F_{u_t}$ that was 
stated in Equation (\ref{eqn:42a}).

In order to go from (\ref{eqn:42a}) to (\ref{eqn:42b}), let us fix
$z,w \in \bC$ such that $|z|, |w|$ satisfy the restrictions 
mentioned above, and let us denote
\[
W( t e^{-t/2} z) =: \alpha, \, W( t e^{-t/2} w) =: \beta.
\]
From the definition of the Lambert function it follows that
$\alpha e^{\alpha} = t e^{-t/2} z, \ \ 
\beta e^{\beta} = t e^{-t/2} w$,
and multiplying together the latter equations gives
\begin{equation}  \label{eqn:42d}
t^2 zw = \alpha \beta e^{t + \alpha + \beta}.
\end{equation}
We write the right-hand side of (\ref{eqn:42a}) in terms of 
$\alpha$ and $\beta$ (where $t^2 zw$ is substituted from 
Equation (\ref{eqn:42d})), and we obtain
\[
F_{u_t} (z,w) = \frac{1}{t} \cdot 
\frac{ \alpha \beta e^{t + \alpha + \beta}- \alpha \beta}{t
                                           + \alpha + \beta}
= \frac{\alpha \beta}{t} \ \int_0^1 
e^{x(t + \alpha + \beta)} \, dx.
\]
Finally, in the latter integral we make the substitution 
$s = 1-x$, which leads to 
\[
F_{u_t} (z,w) 
= \frac{\alpha \beta}{t} \cdot e^{t + \alpha + \beta} \cdot
\int_0^1 
e^{-s(t + \alpha + \beta)} \, ds.
\]
The constant $( \alpha \beta e^{t + \alpha + \beta} )/t$ 
is (by (\ref{eqn:42d})) equal to $tzw$, hence reverting back 
from $\alpha, \beta$ to $z,w$ takes us precisely to the integral 
formula stated in Equation (\ref{eqn:42b}).
\end{proof}

$\ $

\begin{proposition}  \label{prop:4.3}
Let us fix $t \in ( 0, \infty )$ and let $u_t$ be as above
(free unitary Brownian motion at time $t$).  For every 
$k, \ell \in \bN$ we have 
\begin{equation}  \label{eqn:43a}
\kappa_{k+ \ell } ( \,  \underbrace{u_t, \ldots , u_t}_k \, , 
\,  \underbrace{u_t^{*}, \ldots , u_t^{*}}_{\ell} \, ) 
= \frac{(-1)^{k+ \ell}}{ (k-1)! ( \ell -1)! } 
t^{k+ \ell -1} \, ( e^{-t/2} )^{k+ \ell -2} \cdot I_{k,\ell} (t),
\end{equation}
where 
\begin{equation}  \label{eqn:43b}
I_{k, \ell} (t) :=  \int_0^1 e^{-ts} \, s^2 \, 
(s+k-1)^{k-2} \, (s+ \ell -1)^{\ell -2} \, ds.
\end{equation}
\end{proposition}

\begin{proof} 
It is known (see e.g.~\cite{CJK1997}) that for any $s \in [0,1]$
and $y \in \bC$ with $|y| < 1/e$ one has the series expansion
\[
e^{-sW(y)} = 1 - sy + \frac{s(s+2)}{2!} y^2 
- \frac{s(s+3)^2}{3!} y^3 + \cdots 
+ (-1)^n \frac{s(s+n)^{n-1}}{n!} y^n + \cdots ,
\]
which we will find convenient to write concisely as
\begin{equation}   \label{eqn:43c}
e^{-sW(y)} = \sum_{n=0}^{\infty} \frac{s(s+n)^{n-1}}{n!} (-y)^n.
\end{equation}

Let us then pick some $z,w$ with $|z|, |w|$ small enough (in the 
sense discussed in Lemma \ref{lemma:4.2}) and such that moreover
$z,w$ are real negative numbers.  By using the expansion 
(\ref{eqn:43c}) in the Equation (\ref{eqn:42b}) of 
Lemma \ref{lemma:4.2} we infer that
\[
F_{u_t} (z,w) = tzw \, \int_0^1 e^{-ts} 
\cdot \sum_{m=0}^{\infty} \frac{s(s+m)^{m-1}}{m!} (-t e^{-t/2}z)^m 
\cdot \sum_{n=0}^{\infty} \frac{s(s+n)^{n-1}}{n!} (-t e^{-t/2}w)^n 
\, ds
\]
\[
= \int_0^1 \, \Bigl( \, \sum_{m,n = 0}^{\infty} \, tzw \cdot e^{-ts} 
\cdot \frac{s(s+m)^{m-1}}{m!} (-t e^{-t/2}z)^m 
\cdot \frac{s(s+n)^{n-1}}{n!} (-t e^{-t/2}w)^n \, \Bigr) \, ds.
\]
The terms of the infinite double-sum are non-negative, hence
the monotone convergence theorem allows us to interchange the 
double-sum with the integral.  When we do that, and we 
move the powers of $-z,-w,t, e^{-t/2}$ outside the integral, we 
come to the fact that (for $z,w$ picked as above) we have
\[
F_{u_t} (z,w) = \sum_{m,n=0}^{\infty} (-z)^{m+1} (-w)^{n+1} \cdot
\frac{ t^{m+n+1} (e^{-t/2})^{m+n}}{m! n!} \cdot
\Bigl( \, \int_0^1 s(s+m)^{m-1} \, s(s+n)^{n-1} \, ds \Bigr).
\]
It is convenient to also make here the shift of indices 
$m+1 = k$, $n+1 = \ell$, and conclude that
\begin{equation}   \label{eqn:43d}
F_{u_t} (z,w) = \sum_{k, \ell =1}^{\infty} (-z)^k (-w)^{\ell} \cdot
\frac{ t^{k+ \ell -1} (e^{-t/2})^{k+ \ell -2}}{ (k-1)! 
( \ell -1)!} \cdot I_{k, \ell} (t),
\end{equation}
with $I_{k, \ell} (t)$ defined in (\ref{eqn:43b}).

Now, it is easy to see that if we put 
\[
\lambda_{k, \ell} (t) := 
\frac{(-1)^{k+ \ell}}{ (k-1)! ( \ell -1)! } 
t^{k+ \ell -1} \, ( e^{-t/2} )^{k+ \ell -2} \cdot I_{k,\ell} (t),
\ \ k, \ell \in \bN,
\]
then the formula
\[
G_t (z,w) := \sum_{k, \ell = 1}^{\infty} \lambda_{k, \ell} (t) 
z^k w^{\ell}
\]
gives an analytic function defined for $|z|, |w|$ small enough.
Indeed, one can simply bound the integrand in $I_{k, \ell} (t)$ 
by $k^{k-2} \ell^{\ell - 2}$ to conclude that 
\[
0 \leq I_{k, \ell} (t) \leq 
k^{k-2} \ell^{\ell - 2} \leq \gamma \cdot e^k (k-1)! \cdot 
e^{\ell} ( \ell - 1 )!,
\]
where $\gamma >0$ is an absolute constant (not depending on $k, \ell$)
--- the second inequality displayed above follows from Stirling's 
formula.  This implies in turn the bound
\[
| \lambda_{k, \ell} (t) | \leq (\gamma e^t /t) \cdot
(e t e^{-t/2})^{k+ \ell}, \ \ \forall \, k, \ell \in \bN,
\]
and gives the claim about the existence of $G_t (z,w)$.

Finally, Equation (\ref{eqn:43d}) can be read as saying that 
$F_{u_t} (z,w) = G_t (z,w)$ for $z,w$ real negative numbers
of small enough absolute value.  This implies that $F_{u_t}$ 
and $G_t$ must have the same series expansion around $(0,0)$, 
which is exactly the statement that had to be proved.
\end{proof}

$\ $

The formula for cumulants found in Proposition \ref{prop:4.3} can 
be re-phrased as a formula for the corresponding polynomials 
$Z_{\omega}$, as follows.

$\ $

\begin{theorem}   \label{thm:4.4}
Let $k, \ell$ be positive integers.  There exist polynomials
$U_{k, \ell}, V_{k, \ell} \in \bZ [x]$, uniquely determined,
such that
\begin{equation}   \label{eqn:44a}
\left\{  \begin{array}{ccl}
U_{k, \ell} (x) & = &  - x^{k + \ell -1} \int_0^{\infty} e^{-xs} 
\bigl( \,(s+1)^2 (s+k)^{k-2} (s+ \ell )^{ \ell -2} \, \bigr) ds,  \\
               &   &                                              \\
V_{k, \ell} (x) & = & x^{k + \ell -1} \int_0^{\infty} e^{-xs} 
\bigl( \, s^2 (s+k-1)^{k-2} (s+ \ell -1)^{ \ell -2} \, \bigr) ds,
\end{array}  \right.
\ x \in [ 0, \infty ).
\end{equation}
One has
\begin{equation}   \label{eqn:44b}
Z_{ ( \, \underbrace{1, \ldots , 1}_k \, 
         \underbrace{*, \ldots , *}_{\ell} \, ) } (x,y) 
= \frac{ (-1)^{k + \ell} }{ (k-1)! ( \ell -1)! } \, 
\Bigl(  \, U_{k, \ell} (x) y^{k+ \ell} 
      + V_{k, \ell} (x) y^{k+ \ell -2} \, \Bigr).
\end{equation}
\end{theorem}

\begin{proof}  
In order to verify that the function $U_{k, \ell} (x)$ defined 
by the first integral in (\ref{eqn:44a}) is indeed a polynomial,
we expand the product $(s+1)^2 (k+s)^{k-2} ( \ell +s)^{ \ell -2}$
in powers of $s$, then use the fact that 
\[
x^{k + \ell -1} \, \int_0^{\infty} e^{-xs} s^m ds 
= m! \, x^{(k + \ell -1) - (m+1)},
\ \ 0 \leq m \leq k + \ell -2.  
\]
A similar calculation shows that $V_{k, \ell} (x)$ is a polynomial 
as well.

In order to prove that (\ref{eqn:44b}) holds, it suffices to fix 
a $t \in [0, \infty )$ and to verify the following fact: when 
evaluated at $(t, e^{-t/2})$, the polynomial in $(x,y)$ from the 
right-hand side of (\ref{eqn:44b}) yields the free cumulant
$\kappa_{k + \ell} ( u_t, \ldots , u_t, u_t^{*}, \ldots, u_t^{*} )$
(with $k$ entries of $u_t$ and $\ell$ entries of $u_t^{*}$).  By 
comparing this fact against the result of Proposition \ref{prop:4.3}, 
and by doing some obvious simplifications, we see that it is 
actually sufficient to check that
\[
t^{k+ \ell -1} \, I_{k, \ell} (t) 
= U_{k, \ell} (t) \, e^{-t} + V_{k, \ell} (t),
\]
where $I_{k, \ell} (t)$ is the integral defined in Equation
(\ref{eqn:43b}).  But the latter verification is immediately 
obtained when one writes the integral ``$\int_0^1$'' which defines 
$I_{k, \ell} (t)$ as a difference
``$\int_0^{\infty} - \int_1^{\infty}$'' (by using the same 
integrand).  Indeed, the very definition of $V_{k, \ell}$ says
that
\[
t^{k+ \ell -1} \, \int_0^{\infty} e^{-ts} \, s^2 \, 
(s+k-1)^{k-2} \, (s+ \ell -1)^{\ell -2} \, ds = V_{k, \ell} (t),
\]
while on the other hand the change of variable 
$\widetilde{s} = s-1$ gives
\[
t^{k+ \ell -1} \, \int_1^{\infty} e^{-ts} \, s^2 \, 
(s+k-1)^{k-2} \, (s+ \ell -1)^{\ell -2} \, ds 
\]
\[
= t^{k+ \ell -1} \, \int_0^{\infty} e^{-t( \widetilde{s} + 1} 
\, ( \widetilde{s} + 1 )^2 
\, ( \widetilde{s} + k)^{k-2} 
\, ( \widetilde{s} + \ell )^{\ell -2} \, d \widetilde{s},
\]
which is $- e^{-t} U_{k, \ell} (t)$.
\end{proof}

$\ $

\begin{remark}  \label{rem:4.5}
Let us illustrate the explicit writing of the polynomials 
$U_{k, \ell}$ and $V_{k, \ell}$ in the special case $\ell = 1$ 
(this gives, in some sense, the simplest possible example of 
free cumulants of $u_t$ and $u_t^{*}$ that are truly ``joint'').
The formulas defining $U_{k, \ell}$ and $V_{k, \ell}$ become here
\[
U_{k,1} (x) = - x^k  \int_0^{\infty} e^{-xs} (s+1) 
                                             (s+k)^{k-2} ds,  \ \ 
V_{k,1} (x) =   x^k  \int_0^{\infty} e^{-xs} s(s+k-1)^{k-2} ds,
\ \ k \in \bN.
\]
We note the special relation
\[
U_{k,1} = - \frac{1}{k} V_{ k+1,1 }, \ \ 
\forall \, k \geq 1, 
\]
which is easily derived by writing 
$V_{k+1,1} (x) =   -x^k  \int_0^{\infty} ( e^{-xs} )' \cdot
s(s+k-1)^{k-2} ds$, and by doing an integration by parts.
We thus only need to write explicitly the $V_{k,1}$'s; this is 
done in the way shown at the beginning of the preceding proof, 
which gives $V_{1,1} (x) = V_{2,1} (x) = 1$ and
\begin{equation}   \label{eqn:45a}
V_{k,1} (x) = 
\sum_{j=0}^{k-2} \left( \begin{array}{c} k-2 \\ j \end{array} \right) 
\cdot (k-1-j)! \cdot (k-1)^j \ t^j, \ \, k \geq 3.
\end{equation}

In terms of the actual $*$-cumulants of $u_t$, the above 
considerations say that for every $k \in \bN$ we have: 
\begin{align*}
\kappa_{k+1} 
( \, \underbrace{u_t, \ldots , u_t}_k  \, , u_t^{*} \, ) 
& = Z_{ ( \, \underbrace{1, \ldots , 1}_k \, , * \, ) }
    ( t, e^{-t/2} )                                           \\
& = \frac{(-1)^{k+1}}{ (k-1)! }
    \Bigl( \, U_{k,1}(t) (e^{-t/2})^{k+1}
            + V_{k,1}(t) (e^{-t/2})^{k-1} \, \Bigr)            \\
& = \frac{ (- e^{-t/2})^{k-1} }{ (k-1)! }
    \Bigl( \, - \frac{1}{k} V_{k+1,1}(t) e^{-t}
                          + V_{k,1}(t) \, \Bigr).
\end{align*}
Thus, if we consider the sequence of polynomials
\[
V_k := V_{k,1} / (k-1)!, \ \ k \in \bN,
\]
(with $V_{k,1}$ taken from Equation (\ref{eqn:45a})), the 
conclusion is that for every $k \in \bN$ and $t \in [0, \infty )$
we have 
\begin{equation}   \label{eqn:45b}
\kappa_{k+1} 
( \, \underbrace{u_t, \ldots , u_t}_k  \, , u_t^{*} \, ) 
= (- e^{-t/2} )^{k-1} 
\bigl( V_k (t) - e^{-t} V_{k+1} (t) \bigr).
\end{equation}
So for instance for $k \leq 4$ we have
\[
\left\{   \begin{array}{lcl}
\kappa_2 (u_t, u_t^{*}) & =  & (-e^{-t/2})^0 \, (1 -e^{-t}),    \\
\kappa_3 (u_t, u_t, u_t^{*}) & = & (-e^{-t/2})^1 \, 
                     \bigl( \, 1- e^{-t} (t+1) \, \bigr),     \\
\kappa_4 (u_t, u_t, u_t, u_t^{*}) & = & (-e^{-t/2})^2 \, 
\bigl( \, (t+1) - e^{-t} ( \frac{3}{2} t^2 + 2t + 1 ) \, \bigr), \\
\kappa_5 (u_t, u_t, u_t, u_t, u_t^{*}) & = & (-e^{-t/2})^3 \,
\bigl( \, ( \frac{3}{2} t^2 + 2t + 1 ) 
- e^{-t} ( \frac{8}{3} t^3 + 4 t^2 + 3t + 1 ) \, \bigr) .
\end{array}  \right.
\]
\end{remark}

$\ $

$\ $

\section{Another special case --- alternating $\omega$'s}

\setcounter{equation}{0}

The special form of free joint cumulants for a Haar unitary and 
its adjoint (reviewed in Equation (\ref{eqn:1e}) of the introduction)
suggests that in our discussion of the polynomials $Z_{\omega}$
we should consider the case when $\omega$ is an alternating string 
of even length.  The polynomial $Z_{\omega}$ associated to the 
alternating string $(1,*, \ldots , 1,*) \in \{ 1,* \}^{2k}$
is of the form
\begin{equation}  \label{eqn:5a}
(-1)^{k-1} \Bigl( \, C_{k-1} - T_1^{(k)} (x) \, y^2 
+ T_2^{(k)} (x) \, y^4 - \cdots + (-1)^k T_k^{(k)} (x) \, y^{2k}
\, \Bigr),
\end{equation}
where $C_{k-1}$ is the $(k-1)$-th Catalan number, and every 
$T_j^{(k)}$ ($1 \leq j \leq k$) is a polynomial of degree 
$j-1$ with strictly positive rational coefficients.  
Examples for small $k$:
\[
\left\{  \begin{array}{lcl}
Z_{(1,*)} (x,y)   & =  & 1 -  y^2, \\ 
                  &    &                             \\
Z_{ (1,*,1,*) } (x,y) & = & -1 + 4y^2 - (2x+3) y^4,   \\
                  &    &                             \\
Z_{ (1,*,1,*,1,*) } (x,y) & = &
2 - 15 y^2 + (12x+30) y^4 - (6x^2 + 18x + 17) y^6,   \\
                  &    &                             \\
Z_{ (1,*,1,*,1,*,1,*) } (x,y) & = & 
-5 + 56y^2 - 28(2x+7) y^4
              + 8(6x^2 + 26x + 33) y^6               \\
          & &
       - \left(\frac{64}{3}x^3 + 96x^2 + 172x + 119\right) y^8.
\end{array}  \right.
\]
The inductive verification that the pattern (\ref{eqn:5a}) holds 
for general $k$ is not hard (based on the recursion from 
Proposition \ref{prop:3.4}), and is left as exercise to the 
reader.  In this section we do not focus on coefficients of 
$Z_{\omega}$'s, but we find it more interesting to look at the 
actual cumulants 
\begin{equation}  \label{eqn:5b}
\xi_n (t) 
:= \kappa_{2n} (u_t, u_t^{*}, \ldots , u_t, u_t^{*} )
= Z_{ ( \, \underbrace{1,*, \ldots , 1,*}_{2n} \, ) }
    (t, e^{-t/2} ), \ n \geq 1,
\end{equation}
where $u_t$ is the free unitary Brownian motion at time $t$.
The notation introduced in (\ref{eqn:5b}) emphasizes the 
dependence on $t$.  This is of relevance because the 
main point of the section is to put into evidence a recursion 
for the derivative of $\xi_n$ with respect to $t$, as shown 
next.

$\ $

\begin{theorem} \label{thm:5.1}
Let $\xi_n (t)$ be as above.  Then for every
$n \geq 2$ one has
\begin{equation}    \label{eqn:51a}
- \frac{1}{n} \frac{d\xi_n}{dt}(t) 
= \xi_n (t) +  \sum_{m=1}^{n-1} \xi_m (t) \xi_{n-m} (t),
  \ \ t \in [ 0, \infty ).
\end{equation}
\end{theorem}

\begin{proof}
For convenience of notation, throughout this proof we will 
fix a tracial $*$-probability space $( \cA , \varphi )$
which is large enough to contain all the unitaries $u_t$
for $t \in [0 , \infty )$.  By enlarging $( \cA , \varphi )$ a
bit\footnote{ For instance we can replace $( \cA , \varphi )$ by 
the free product $( \cA , \varphi ) * ( L^{\infty} [0,1], dt )$,
and take 
$p_{\theta}, q_{\theta} \in L^{\infty} [0,1], dt )$ to be the
indicator functions of the intervals 
$[0, \theta]$ and $[1- \theta , 1]$, respectively. }
further, we will moreover assume that $\cA$ contains two 
families of elements
$\{ p_{\theta} \mid 0 < \theta < 1/2 \}$ and
$\{ q_{\theta} \mid 0 < \theta < 1/2 \}$ 
such that

(i) $p_{\theta}^2 = p_{\theta}^{*} = p_{\theta}, \ 
     q_{\theta}^2 = q_{\theta}^{*} = q_{\theta}$ and
$p_{\theta} q_{\theta} = q_{\theta} p_{\theta} = 0,
     \ \ \forall \, \theta \in (0, 1/2)$;

(ii) $\varphi ( p_{\theta} ) = \varphi ( q_{\theta} ) = \theta, 
     \ \ \forall \, \theta \in (0, 1/2)$;

(iii) $\{ p_{\theta} , q_{\theta} \}$ is free from
$\{ u_t, u_t^{*} \}$, for all
$\theta \in (0, 1/2)$  and $t \in [ 0, \infty )$.

\vspace{6pt}

We consider the rescaled elements $v_t = e^{t/2} u_t$. 
Following \cite{BGL2011}, for every $n \in \bN$ we define a 
function $f_{2n} : [0, \infty ) \times ( 0, 1/2 ) \to \bR$ by
\begin{equation}   \label{eqn:51b}  
f_{2n} ( t, \theta ) := \varphi \bigl( \, (p_{\theta} \, v_t
\, q_{\theta} \, v_t^{*} )^n \, \bigr),  \ \ \forall \, 
t \geq 0 \mbox{ and } 0 < \theta < 1/2.
\end{equation}
[For instance for $n = 1$ we have
$f_2 (t, \theta ) 
:= \varphi ( p_{\theta} \, v_t \, q_{\theta} \, v_t^{*} )$, 
and an immediate application of formula (\ref{eqn:24a}) for 
alternating moments yields
$f_2 (t, \theta ) = \theta^2 ( e^t - 1)$.]

\vspace{6pt}

{\em Claim.} For every $n \in \bN$, the function $f_{2n}$
is of the form
\begin{equation}   \label{eqn:51c}  
f_{2n} ( t, \theta ) = \sum_{j=1}^{2n} g_{n,j} (t) 
\, \theta^j, 
\end{equation}
where the $g_{n,j}$'s are quasi-polynomials, and where (for
$j=2n$) we have 
\begin{equation}   \label{eqn:51d}  
g_{n,2n} (t) = e^{nt} \xi_n (t).
\end{equation}

\vspace{6pt}

{\em Verification of Claim.}  Fix $n \in \bN$ for which we will
verify that (\ref{eqn:51c}) and (\ref{eqn:51d}) hold.  We write
$f_{2n} ( t, \theta )$ as
$\varphi ( \, v_t \, q_{\theta} \, v_t^{*} \, p_{\theta} 
\cdots v_t \, q_{\theta} \, v_t^{*} \, p_{\theta} \, )$, 
and we expand this alternating moment of order $4n$ in the way 
indicated in Remark \ref{rem:2.4}.1, in terms of moments of 
$p_{\theta},q_{\theta}$ and of free cumulants of $v_t, v_t^{*}$.
In this way we obtain the formula
\begin{equation}    \label{eqn:51e}
f_{2n} ( t, \theta )
= \sum_{\sigma \in NC(2n)} g_{\sigma} (t) \cdot 
                           h_{\sigma} ( \theta ),
\end{equation}
where for $\sigma \in NC(2n)$ we put
\[
\left\{  \begin{array}{ccl}
g_{\sigma} (t) & = & 
\prod_{V \in \sigma} \kappa_{|V|} \bigl( \, ( v_t, v_t^{*}, 
\ldots , v_t, v_t^{*}) \mid V \, ) \, \bigr),                    \\
               &   &                                             \\
h_{\sigma} ( \theta ) & = & 
\prod_{W \in \Kr_{2n} ( \sigma )} \varphi_{|W|} 
\bigl( \, (q_{\theta}, p_{\theta}, \ldots , 
q_{\theta}, p_{\theta} ) \mid W \, ) \ \bigr).
\end{array}   \right.
\]
Note that every $g_{\sigma}$ can be written as
\[
g_{\sigma} (t) = e^{nt} \cdot
\prod_{V \in \sigma} \kappa_{|V|} \bigl( \, ( u_t, u_t^{*}, 
\ldots , u_t, u_t^{*}) \mid V \, ) \, \bigr), 
\]
and is thus a quasi-polynomial by Proposition 3.1.

Let us next observe that for every non-empty set 
$W \subseteq \{ 1, \ldots , 2n \}$, the moment 

\noindent
$\varphi_{|W|} \bigl( \, (q_{\theta}, p_{\theta}, \ldots , 
q_{\theta}, p_{\theta} ) \mid W \, ) \ \bigr)$ is equal to either 
$0$ or $\theta$.  Indeed, if $W$ contains both odd and even numbers,
then the moment in discussion vanishes due to the hypothesis that 
$p_{\theta} q_{\theta} = q_{\theta} p_{\theta} = 0$.
In the opposite case, we are looking either at
$\varphi ( p_{\theta}^{|W|} )$ or at $\varphi ( q_{\theta}^{|W|} )$, 
and both these moments are equal to $\theta$.

The observation from the preceding paragraph implies that, for 
every $\sigma \in NC(2n)$, the value of
$h_{\sigma} ( \theta )$ is either $0$ or $\theta^{j( \sigma )}$,
with $j( \sigma ) := \mid \, \Kr_{2n} ( \sigma ) \, \mid
\ = \ (2n+1) - | \sigma |$,
where $| \sigma |$ denotes the number of blocks of $\sigma$.  
Moreover, the case 
$h_{\sigma} ( \theta ) =  \theta^{j( \sigma )}$ occurs if and 
only if every block $W$ of $\Kr_{2n} ( \sigma )$ either is 
contained in $\{ 1,3, \ldots , 2n-1 \}$ or is contained in
$\{ 2,4, \ldots , 2n \}$.  The latter condition on 
$\Kr_{2n} ( \sigma )$ is easily seen to be equivalent to the 
fact that every block of $\sigma$ has even cardinality (cf.
\cite{NS2006}, Exercise 9.42(1) on p. 154).  We thus come to
the conclusion that we can re-write Equation (\ref{eqn:51e})
in the form
\[
f_{2n} ( t, \theta )
= \sum_{j=1}^{2n} g_{n,j} (t) \cdot \theta^j ,
\]
where for $1 \leq j \leq 2n$ we define the quasi-polynomial 
$g_{n,j}$ to be
\begin{equation}   \label{eqn:51f}
g_{n,j} := \sum_{  \begin{array}{c}
{\scriptstyle \sigma \in NC (2n), \ | \sigma | = 2n+1-j} \\
{\scriptstyle and \ |V| \ even \ for \ all \ V \in \sigma}
\end{array}  } \ g_{\sigma} (t).
\end{equation}
In the special case $j = 2n$, the only partition involved in the 
sum from (\ref{eqn:51f}) is $\sigma = 1_{2n}$, which has
$g_{ { }_{1_{2n}} } (t) = 
\kappa_{2n} (v_t, v_t^{*}, \ldots , v_t, v_t^{*} )
= e^{nt} \cdot
\kappa_{2n} (u_t, u_t^{*}, \ldots , u_t, u_t^{*} )$,
and (\ref{eqn:51d}) also follows.
\hfill {\em [End of Verif. of Claim]}

\vspace{10pt}

Besides the $f_2, f_4, \ldots , f_{2n}, \ldots$ 
introduced in (\ref{eqn:51c}) we consider, also following 
\cite{BGL2011}, the function 
$f_0 : [0, \infty ) \times ( 0, 1/2 ) \to \bR$ defined by
\[
f_0 ( t, \theta ) := \theta, \ \ 
\forall \, t \geq 0 \mbox{ and } 0 < \theta < 1/2.
\]
(Note that the definition of $f_0$ is not obtained by extending
the range of $n$ from $\bN$ to $\bN \cup \{ 0 \}$ in 
Equation (\ref{eqn:51b})!)  Theorem 3.4 in \cite{BGL2011} 
gives us that for every $n \geq 1$, the partial derivative 
$\partial_t f_{2n}$ satisfies the following recursion: 
\[
\partial_t f_{2n} (t, \theta) 
= - \sum_{ \begin{array}{c} 
{\scriptstyle 1 \leq k < \ell \leq 2n}  \\
{\scriptstyle k = \ell \ mod \ 2}
\end{array} } \ f_{2n-( \ell -k )} (t, \theta)
\, f_{\ell -k} (t, \theta)
\]
\begin{equation}  \label{eqn:51g}
+ e^t \  \sum_{ \begin{array}{c} 
{\scriptstyle 1 \leq k < \ell \leq 2n}  \\
{\scriptstyle k \neq \ell \ mod \ 2}
\end{array} } \ f_{2n-( \ell -k )-1} (t, \theta)
\, f_{\ell -k-1} (t, \theta).
\end{equation}
[For illustration we record that the special cases $n=1$ and $n=2$ 
of (\ref{eqn:51g}) come to 
$\partial_t f_2 = e^t f_0^2$, and respectively to 
$\partial_t f_4 = - 2 f_2^2 + e^t \cdot 4 f_0 f_2$.]

For a fixed $t \in [0, \infty )$, both sides of 
Equation (\ref{eqn:51g}) are polynomials of degree $2n$ in 
$\theta$; so it makes sense to extract the coefficient of 
$\theta^{2n}$ in this equation.  On the left-hand side, the 
coefficient of $\theta^{2n}$ is equal to the derivative 
of $g_{n,2n} (t)$, thus to
\begin{equation}  \label{eqn:51h}
e^{nt} \cdot \bigl( n \xi_n (t) + \frac{d \xi_n}{dt} (t) \bigr).
\end{equation}  
On the right-hand side of (\ref{eqn:51g}) only the terms from 
the first of the two sums contribute to $\theta^{2n}$, giving a 
coefficient equal to 
\[
- \sum_{ \begin{array}{c} 
{\scriptstyle 1 \leq k < \ell \leq 2n}  \\
{\scriptstyle k = \ell \ mod \ 2}
\end{array} } \ \bigl( e^{(n-( \ell -k )/2)t} 
                       \xi_{n - ( \ell -k)/2}(t) \bigr) \cdot
              \bigl( e^{(( \ell -k )/2)t} 
                       \xi_{( \ell -k)/2}(t) \bigr) 
= - e^{nt} \cdot n \, \sum_{m=1}^{n-1} \xi_m (t) \xi_{n-m} (t).
\]
When we equate the latter quantity with the one in
(\ref{eqn:51h}), formula (\ref{eqn:51a}) follows. 
\end{proof}

$\ $

\begin{corollary}   \label{cor:5.2}
Consider the function 
\begin{equation}   \label{eqn:52a}
H(t,z) := \frac{1}{2} + \sum_{n = 1}^{\infty} \xi_n (t) z^n
\end{equation}
defined on $\{ (t,z) \mid t \in [ 0, \infty ), \, z \in \bC,
\, |z| < 1/16^2 \}$.  Then $H$ satisfies the partial 
differential equation
\begin{equation}   \label{eqn:52b}
\partial_t H + 2z \, H \, \partial_z H = z,
\end{equation}
with initial condition $H(0,z) = 1/2$.
\end{corollary}

\begin{proof}  The domain of $H$ is considered by taking into
account the bounds for $\xi_n (t)$'s that follow from Remark
\ref{rem:2.4}.4. In order to obtain (\ref{eqn:52b}), we square
both sides of (\ref{eqn:52a}) and then we take partial 
derivative $\partial_z$, to find that
\[
z \cdot \partial_z H^2 (t,z) 
= \xi_1 (t) z 
+ 2 \Bigl( \xi_2 (t) + \xi_1^2 (t) \Bigr) z^2 + \cdots
+ n  \Bigl( \xi_n (t) + \sum_{m=1}^{n-1} \xi_m (t) \xi_{n-m} (t) 
     \Bigr) z^n + \cdots 
\]
The latter equation can be written as 
\begin{align*}
z \cdot \partial_z H^2 (t,z) 
& = \xi_1 (t) z - \xi_2 ' (t) z^2 - \cdots  
                - \xi_n ' (t) z^n - \cdots 
    \ \ \mbox{ $\ $ (by Theorem \ref{thm:5.1})}             \\
& = (1- \xi_1 ' (t)) z - \xi_2 ' (t) z^2 - \cdots  
                       - \xi_n ' (t) z^n - \cdots 
    \ \mbox{ (because $\xi_1 (t) = 1 - e^{-t}$)}            \\
& = z - \partial_t H(t,z),
\end{align*}
and (\ref{eqn:52b}) follows.  The condition on $H(0,z)$
is also clear, since $\xi_n (0) = 0$ for all $n \in \bN$.
\end{proof}

$\ $

\begin{remark}  \label{rem:5.3}
$1^o$ Starting from $\xi_1 (t) = 1 - e^{-t}$ and the initial 
condition $\xi_n (0) = 0, \ \forall \, n \geq 2$, one can use
Theorem \ref{thm:5.1} to calculate all the $\xi_n$'s, getting
$\xi_2 (t) = -1 + 4e^{-t} - (2t+3)e^{-2t}$, then
$\xi_3 (t) = 2 - 15e^{-t} + 6(2t+5)e^{-2t}$
$- (6t^2 + 18t + 17)e^{-3t}$, and so on. 

\vspace{6pt}

$2^o$  It stands to reason that one should also look for a
description of the functions $\xi_n (t)$ that is done by plain
algebra (without resorting to the derivative $\frac{d}{dt}$), 
for a given value of $t$.  That is, we are interested in an
algebraic description for the function
\begin{equation}  \label{eqn:53b}
H_t : \{ z \in \bC \mid \ |z| < 1/16^2 \} \to \bC, 
\ \ \mbox{$\ $} \ \ 
H_t (z) := H(t,z) = \frac{1}{2} 
+ \sum_{n=1}^{\infty} \xi_n (t) z^n.
\end{equation}
We will achieve this by examining the characteristic curves of
the p.d.e.~found in Corollary \ref{cor:5.2}.

In order to state precisely what is the algebraic description 
obtained for $H_t$, we introduce an auxiliary complex parameter
$c$ and we consider, for every $t \in [0, \infty )$, the function
\begin{equation}  \label{eqn:53c}
\FF_t : \Omega_t \to \bC, \ \ \mbox{$\ $} \ 
\FF_t (c) := \frac{ c^2 (1-c^2) e^{ct} }{ ( 
\, (1+c) - (1-c) e^{ct} \, )^2 },              
\end{equation}
where $\Omega_t$ is the open set
$\{ c \in \bC \mid 1+c \neq (1-c) e^{ct} \}$.
One has $\Omega_t \ni 1$, with $\FF_t (1) = 0$ and 
$\FF_t' (1) = e^{-t/2} \neq 0$.  The inverse function theorem 
thus gives a $\delta_t > 0$ such that an analytic inverse
for $\FF_t$ can be defined on 
$\{ z \in \bC \mid \ |z| < \delta_t \}$, sending $0$ back 
to $1$.  We denote this compositional inverse as 
\begin{equation}  \label{eqn:53d}
\FF_t^{ \langle -1 \rangle } : 
\{ z \in \bC \mid \ |z| < \delta_t \} \to \bC. 
\end{equation}
$\FF_t^{ \langle -1 \rangle }$ is injective and its range-set
$\mbox{Ran} ( \, \FF_t^{ \langle -1 \rangle } \, )$ is an open 
subset of $\Omega_t$.

Without loss of generality, we may assume that in (\ref{eqn:53d})
we have $\delta_t < 1/16^2$, so that $H_t (z)$ from Equation 
(\ref{eqn:53b}) is sure to be defined, too, for $|z| < \delta_t$.
\end{remark}

$\ $

\begin{theorem} \label{thm:5.4}
Let $t \in [0, \infty )$ be fixed, and consider the analytic 
functions $H_t$ and $\FF_t^{ \langle -1 \rangle }$ defined in 
Remark \ref{rem:5.3}.2.  Then one has
\begin{equation}  \label{eqn:54a}
[ \, H_t (z) \, ]^2 = z + \frac{1}{4}
[ \, \FF_t^{ \langle -1 \rangle } (z) \, ]^2,
\ \ |z| < \delta_t.
\end{equation}
\end{theorem}

\begin{proof}  Our strategy will be to prove the following fact.
\begin{equation}  \label{eqn:54b}
\left[  \begin{array}{l}
\mbox{There exists $\ee_t > 0$ such that for every 
       $c \in ( 1 - \ee_t, 1 + \ee_t ) \subseteq \bR$
       one has:}                                          \\
                                                          \\
\mbox{ $\ $ $\rightarrow \ c \in \Omega_t$ and
$| \FF_t (c) | < 1/16^2$
(hence $H_t( \, \FF_t (c) \, )$ is defined); }              \\
                                                          \\
\mbox{ $\ $ $\rightarrow$ \ 
$[  H_t (  \, \FF_t (c) \, ) ]^2  
             = \FF_t (c) + \frac{c^2}{4}$. }
\end{array}  \right.
\end{equation}
This fact implies the statement of the theorem.  Indeed, let us 
assume that (\ref{eqn:54b}) holds.  Take a strictly decreasing
sequence $(c_n)_{n=1}^{\infty}$ in 
$( 1 - \ee_t, 1 + \ee_t ) \cap
\mbox{Ran} ( \, \FF_t^{ \langle -1 \rangle } \, )$,
with $\lim_{n \to \infty} c_n = 1$, and put 
$z_n := \FF_t (c_n)$, $n \in \bN$. Then $(z_n)_{n=1}^{\infty}$ are
distinct points in $\{ z \in \bC \mid \ |z| < \delta_t \}$, with 
$\lim_{n \to \infty} z_n = 0$, and by applying the last line
of (\ref{eqn:54b}) to the $c_n$ we get
\[
[ H_t (z_n) ]^2 =  z_n + \frac{1}{4} 
[ \FF_t^{ \langle -1 \rangle } (z_n) ]^2, \ \ 
\forall \, n \in \bN .
\]
Hence the analytic functions appearing on the two sides of
(\ref{eqn:54a}) coincide on a subset of
$\{ z \in \bC \mid \ |z| < \delta_t \}$ which has $0$ as 
accumulation point, and (\ref{eqn:54a}) follows.

\vspace{6pt}

We now start towards the proof of the fact stated in
(\ref{eqn:54b}).  We consider the rectangular strip
\[
R := [ 0, \infty ) \times 
( - 1/16^2 , + 1/16^2 ) \subseteq \bR^2,
\]
and we consider the restriction of $H$ (from its domain stated
in Corollary \ref{cor:5.2}) to $R$.  This restriction will still
be denoted as $H$, and
\footnote{ Since ``$t$'' is here a specific time that was fixed 
in the statement of the theorem, we will use the generic letter 
``$s$'' for the first component of a point in $R$. }
we put
\[
\Gamma := \{ (s,x,u) \mid (s,x) \in R, \ 
u = H(s,x) \in \bR \} \ \ \mbox{ (graph of restricted $H$). }
\]
We also consider the vector field $V : R \times \bR \to \bR^3$ 
defined by
\begin{equation}   \label{eqn:54c}
V(s,x,u) := (1, 2xu, x), \ \ \mbox{ for $(s,x) \in R$,
$u \in \bR$. }
\end{equation}
The partial differential equation (\ref{eqn:52a}) says that 
for every $(s,x) \in R$, the vector 
$V \bigl( \, s,x,H(s,x) \, \bigr)$ is orthogonal to the 
normal direction $\bigl( \, ( \partial_t H) (s,x), 
( \partial_x H) (s,x), -1 \bigr)$ to $\Gamma$ at the point 

\noindent
$\bigl( \, s,x,H(s,x) \, \bigr)$.  It follows that 
$V \bigl( \, s,x,H(s,x) \, \bigr)$ gives a tangent direction 
to the graph $\Gamma$, at the point 
$\bigl( \, s,x,H(s,x) \, \bigr)$.

We next pick an $a \in ( - 1/16^2, + 1/16^2 )$ and we consider
a path (a.k.a. characteristic curve) 
$L_a : [ 0, \beta (a) ) \to \bR^3$ which has 
\begin{equation}  \label{eqn:54d}
L_a (0) = ( 0,a , 1/2) \in \Gamma
\end{equation}
and follows the vector field $V$:
\begin{equation}  \label{eqn:54e}
L_a ' (s) = V( \, L_a (s) \, ), \ \ \forall \, 
s \in [ 0, \beta (a) ).
\end{equation}
When we write $L_a$ componentwise,
\[
L_a (s) = ( p_a (s), q_a (s), r_a (s) ), \ \ 
0 \leq s < \beta (a),
\]
the Equation (\ref{eqn:54e}) becomes a system of ordinary 
differential equations, for which (\ref{eqn:54d}) gives an 
initial condition:
\begin{equation}  \label{eqn:54f}
\left\{  \begin{array}{l}
p_a ' (s) = 1, \ q_a ' (s) = 2 q_a (s) \, r_a (s), 
\ r_a ' (s) = q_a (s),                                \\
                                                      \\
\mbox{ $\ $ with $p_a (0) = 0, \, q_a (0) = a, \,
                  r_a (0) = 1/2$. }                 \\
\end{array}  \right.
\end{equation}
Luckily, the Cauchy problem from (\ref{eqn:54f}) can be solved
explicitly.  More precisely: considering the auxiliary\footnote{ It is useful to keep in mind that $c$ runs in a 
neighbourhood of $1$ (it satisfies
$\sqrt{63}/8 < c < \sqrt{65}/8$), while $\alpha$ runs in a
neighbourhood of $0$ (has $\sign ( \alpha ) = \sign (a)$ and
$| \alpha | < 4 |a| < 1/64$). }
constants
\begin{equation}  \label{eqn:54g}
c = \sqrt{1-4a}, \ \ 
\alpha = \frac{1-c}{1+c} = \frac{4a}{ (1+c)^2 },
\end{equation}
we get
\begin{equation}  \label{eqn:54h}
p_a (s) = s, \ 
q_a (s) = \frac{ c^2 \alpha e^{cs} }{ ( 1- \alpha e^{cs} )^2 },
\ r_a (s) = \frac{c}{2} \cdot
\frac{ 1 + \alpha e^{cs} }{ 1- \alpha e^{cs} },
\end{equation}
for $0 \leq s < \beta (a)$.  A significant detail which comes 
up while solving (\ref{eqn:54f}) (and can, of course, be checked
directly on (\ref{eqn:54h})) is that one has
\begin{equation}  \label{eqn:54i}
q_a (s) - \bigl( \, r_a (s) \, \bigr)^2 
= a - \frac{1}{4}, \ \ \forall \, s \in [0, \beta (a) ).
\end{equation}

It is quite useful if at this point we take a moment to assess 
what we want to have for ``$\beta (a)$'' in the discussion from 
the preceding paragraph. Clearly, $\beta (a)$ must be in any case
picked such that 

\begin{center}
(i) $1 - \alpha e^{cs} > 0, \ \ \forall 
\, s \in [ 0, \beta (a) )$,
\hspace{0.5cm} and
\hspace{0.5cm} 
(ii) $\vline \ \frac{ c^2 \alpha e^{cs} }{ (1 -
                       \alpha e^{cs} )^2 }   \ \vline \
< \frac{1}{16^2}, \ \forall 
\, s \in [ 0, \beta (a) )$.
\end{center}

\noindent
The condition (i) ensures that the formulas (\ref{eqn:54h}) 
give indeed a well-defined path 
$L_a (s) = (p_a (s), q_a (s), r_a (s))$,
$0 \leq s < \beta (a)$; then (ii) ensures that 
$L_a (s) \in R \times \bR$ (hence that ``$V( L_a (s) )$''
makes sense) for every $s \in [0, \beta (a))$.  We will 
moreover insist that

\vspace{6pt}

(iii) $\beta (a)$ is a continuous function of 
$a \in ( - 1/16^2 , + 1/16^2 )$,

\noindent
and

(iv) $\beta (0) > t$ ( = the time fixed in the statement of
the theorem).

\vspace{6pt}
We leave it as a routine (though tedious) exercise to the 
reader to check that all the conditions (i)--(iv) are fulfilled
if we go with 
\[
\beta (a) = \min \bigl( t+1, 
\beta_{\mathrm{(i)}} (a), 
\beta_{\mathrm{(ii)}} (a) \bigr) \ \mbox{ for }
|a| < 1/16^2,
\]
where $\beta_{\mathrm{(i)}},
\, \beta_{\mathrm{(ii)}} : ( - 1/16^2, + 1/16^2 ) \to 
[0, \infty ]$ are continuous functions describing the natural 
bounds up to which an $s \in [ 0, \infty )$ fulfills the 
conditions (i) and (ii), respectively.  For instance 
$\beta_{\mathrm{(i)}} (a)$ comes out as 
\[
\beta_{\mathrm{(i)}} (a) = \left\{ \begin{array}{cl}
\frac{1}{c} \mbox{ln} \frac{1}{\alpha}, 
                             & \mbox{if $\alpha > 0$,}    \\
                             &                            \\
\infty ,                     & \mbox{if $\alpha \leq 0$,} \\
\end{array}  \right.
\]
with $c = c(a)$ and $\alpha = \alpha (a)$ as defined in 
Equations (\ref{eqn:54g}).

We now invoke a basic result from the theory of quasi-linear
partial differential equations, which states that: since it 
starts at a point $L_a (0) \in \Gamma$, the characteristic 
curve $L_a$ cannot leave the graph $\Gamma$ of $H$.  (See 
e.g.~the theorem on page 10 of \cite{J1978}.)  In other 
words, one has
\begin{equation}   \label{eqn:54k}
H( p_a (s), q_a (s) ) = r_a (s), \ \ \forall
\, a \in ( - 1/16^2, + 1/16^2 ) \mbox{ and }
s \in [ 0, \beta (a) ).
\end{equation}
If we square both sides of (\ref{eqn:54k}) and take into account
the formula (\ref{eqn:54i}) (also the fact that $p_a (s) = s$), 
we arrive to
\begin{equation}   \label{eqn:54l}
[ \, H( s, q_a (s) ) \, ]^2 = q_a (s) + \bigl(
\frac{1}{4} - a \bigr),
\ \ \forall
\, a \in ( - 1/16^2, + 1/16^2 ) \mbox{ and }
s \in [ 0, \beta (a) ).
\end{equation}

Finally, let us return to the time $t \in [0, \infty )$ that was 
fixed in the statement of the theorem.  Since $\beta$ is continuous
and has $\beta (0) > t$, we can find $0 < \lambda_t < 1/16^2$
such that $\beta (a) > t$ for all 
$a \in ( - \lambda_t , \lambda_t )$.  For $|a| < \lambda_t$ we 
can thus put $s=t$ in (\ref{eqn:54l}), to obtain that
\begin{equation}   \label{eqn:54m}
[ \, H ( t, q_a (t) ) \, ]^2 = q_a (t) + \bigl(
\frac{1}{4} - a \bigr).
\end{equation}

On the other hand, we make the following claim.

\vspace{6pt}

{\em Claim.}  If $|a| < \lambda_t$, then $c := \sqrt{1-4a}$
belongs to the domain $\Omega_t$ of the function $\FF_t$, and 
one has $q_a (t) = \FF_t (c)$.

{\em Verification of Claim.} We have 
$t < \beta (a) \leq \beta_{\mathrm{(i)}} (a)$, hence
(from how $\beta_{\mathrm{(i)}} (a)$ is defined) we get 
$1 - \alpha e^{ct} > 0$, where $\alpha = (1-c)/(1+c)$.  This
implies $(1+c) - (1-c) e^{ct} > 0$, and it follows that 
$c \in \Omega_t$.  The equality $q_a (t) = \FF_t (c)$ is then
immediately obtained by comparing the formulas which describe
$q_a (t)$ and $\FF_t (c)$ (cf.~Equation (\ref{eqn:53c}) and the
case $s=t$ of (\ref{eqn:54h})).

\hfill {\em [End of Verif. of Claim]}

\vspace{6pt}

By using the above claim, we convert Equation (\ref{eqn:54m})
into 
\[
\left\{  \begin{array}{l}
[ H_t ( \, \FF_t (c) ) \, ) ]^2 = \FF_t (c) + 
\frac{c^2}{4}, \ \ \mbox{ with } c = \sqrt{1 - 4a},      \\
                                                         \\
\mbox{ $\ $ for every $a \in ( - \lambda_t , \lambda_t )$. }
\end{array}  \right.
\]
But when $a$ runs in $( - \lambda_t, \lambda_t )$, the 
quantity $c = \sqrt{1-4a}$ covers
$( \sqrt{1 - 4 \lambda_t},  \sqrt{1+ 4 \lambda_t} \, )$, 
which contains an open interval centered at $1$.  This implies
the fact stated in (\ref{eqn:54b}), and concludes the proof.
\end{proof}

$\ $

\begin{remark}  \label{rem:5.5}
The formula (\ref{eqn:54a}) from the preceding theorem can be 
used to calculate the alternating cumulants $\xi_n (t)$ without 
doing a derivative $\frac{d}{dt}$, but rather by starting from 
the Taylor expansion around $1$ of the function $\FF_t (c)$ 
defined in Remark \ref{rem:5.3}.2:
\begin{align*}
\FF_t (c) 
& = (c-1) \, \FF_t ' (1) 
+ \frac{(c-1)^2}{2} \, \FF_t '' (1) 
+ \cdots                                                  \\
&                                                         \\
& = ( - \frac{1}{2} e^t ) \, (c-1) 
+ \bigl( \frac{1}{2} e^{2t} - ( \frac{3}{4} + \frac{t}{2} ) e^t 
  \bigr) \, (c-1)^2 + \cdots
\end{align*}
Indeed, considering the expansion 
$\FF_t^{ \langle -1 \rangle } (z) 
= 1 + \lambda_1 z + \lambda_2 z^2  + \cdots $ 
of $\FF_t^{ \langle -1 \rangle }$ around $0$, one can then calculate 
recursively the $\lambda_n$ by writing that
\begin{equation}   \label{eqn:55a}
z = \FF_t \bigl( \, \FF_t^{ \langle -1 \rangle } (z) \, \bigr)
= ( - \frac{1}{2} e^t ) \cdot
\bigl( \, \FF_t^{ \langle -1 \rangle } (z) - 1 \, \bigr)
+ \bigl( \frac{1}{2} e^{2t} - ( \frac{3}{4} + \frac{t}{2} ) e^t 
  \bigr) \cdot
\bigl( \, \FF_t^{ \langle -1 \rangle } (z) - 1 \, \bigr)^2
+ \cdots 
\end{equation}
\[
= ( - \frac{1}{2} e^t ) 
( \lambda_1 z + \lambda_2 z^2 + \cdots  \, )
+ \bigl( \frac{1}{2} e^{2t} - ( \frac{3}{4} + \frac{t}{2} ) e^t 
  \bigr) \, 
( \lambda_1 z + \lambda_2 z^2 + \cdots  \, )^2 + \cdots \, ,
\]
and by identifying coefficients.  The $\lambda_n$ come out as
quasi-polynomials in $-t$ (for instance, as immediately seen
from the few terms recorded in (\ref{eqn:55a}), one gets
$\lambda_1 = -2 e^{-t}$ and 
$\lambda_2 = 4 e^{-t} - (4t+6) e^{-2t}$).

Finally, Equation (\ref{eqn:54a}) says that
\[
\frac{1}{4} + \xi_1 (t) z 
+ ( \xi_2 (t) + \xi_1^2 (t) ) z^2
+ ( \xi_3 (t) + 2 \xi_1 (t) \xi_2 (t) ) z^3 + \cdots 
\]
\[
= z + \frac{1}{4} \Bigl( \, 1 + 2 \lambda_1 z 
+ ( 2 \lambda_2 + \lambda_1^2 ) z^2
+ ( 2 \lambda_3 + 2 \lambda_1  \lambda_2 ) z^3 + \cdots
\, \Bigr),
\]
which allows the recursive calculation of the $\xi_n (t)$ 
(e.g.~$\xi_1 (t) = 1 + \frac{\lambda_1}{2} = 1- e^{-t}$, then
\begin{align*}
\xi_2 (t) 
& = \frac{1}{4} ( 2 \lambda_2 + \lambda_1^2 ) - \xi_1 (t)^2  \\
& = (2 e^{-t} - (2t+2) e^{-2t} ) - (1- e^{-t} )^2
  = -1 + 4e^{-t} - (2t+3) e^{-2t},
\end{align*}
which agrees, of course, with the formulas stated in 
Remark \ref{rem:5.3}.1).
\end{remark}

$\ $

\begin{remark} \label{rem:5.6}
The proof presented above for Theorem \ref{thm:5.4} is a
standard application of the method of characteristics, and
has in its favour the fact that the relevant function $\chi_t$ 
from Equation (\ref{eqn:53c}) is ``discovered'' as we 
move through the argument.  The referee to the paper pointed 
out to us how an alternative, shorter proof of the theorem 
can be made by starting from the observation that, for fixed 
$c$, the function $t \mapsto \chi_{t} (c)$ satisfies the 
differential equation 
\begin{equation}    \label{eqn:56a}
\partial_t \chi_t (c) = 2 \chi_t (c) \,
\Bigl( \, \chi_t (c) + \frac{c^2}{4} \, \Bigr)^{1/2}.
\end{equation}
The present remark gives a sketch of this alternative argument.

For the convenience of having all our functions defined around
the origin, let us consider the shifted sets
$\widetilde{\Omega_t} := \{ c-1 \mid c \in \Omega_t \}$, 
$t \geq 0$, and let us define 
\begin{equation}    \label{eqn:56b}
\chitild (t,z) = \chitild_t (z) := \chi_t (z+1), \ \ 
t \geq 0, \ z \in \widetilde{\Omega}_t.
\end{equation}
Formula (\ref{eqn:56a}) then gives
a family of ordinary differential equations satisfied by the 
functions $t \mapsto \chitild (t,z)$ (with the parameter $z$
running in a neighbourhood of $0$), namely 
\begin{equation}    \label{eqn:56c}
\left\{   \begin{array}{l}
\partial_t \chitild (t,z) = 2 \chitild (t,z) \,
\Bigl( \, \chitild (t,z) + \frac{(z+1)^2}{4} \, \Bigr)^{1/2}, \\
\mbox{ with initial condition 
       $\chitild (0,z) = (1- (z+1)^2)/4$. }
\end{array}   \right.
\end{equation}

Now, Theorem \ref{thm:5.4} can be recast as the statement that a 
certain function constructed out of $H$ is equal to $\chitild$.
Indeed, the conclusion of the theorem can be rewritten as
\begin{equation}    \label{eqn:56d}
 \left[ 4 ( H_t(z)^2 - z ) \right]^{1/2} -1  
= \chi_t^{ \langle -1 \rangle } (z) - 1
= \chitild_t^{ \langle -1 \rangle } (z).
\end{equation} 
So then let us denote 
\begin{equation}   \label{eqn:56d2}
G(t,z) = G_t (z) 
:= \left[ 4 (H_t(z)^2-z) \right]^{1/2} -1.
\end{equation}
By starting from the explicit series expansion of $H_t (z)$
in (\ref{eqn:53b}) and by following through the algebra,
one finds an explicit series expansion for $G_t (z)$,
\[
G_t (z) = 2( \xi_1 (t) - 1 ) z
+ 2( \xi_2 (t) + 2 \xi_1 (t) - 1 ) z^2 + \cdots
\]
What interests us here is that the expansion of $G_t (z)$
has no constant term, and has linear term 
$2( \xi_1 (t) - 1) = 2 \bigl( (1 - e^{-t}) - 1 \bigr) 
= -2 e^{-t} \neq 0$; this implies that $G_t$ is
invertible under composition.  We can therefore define a 
function $K(t,z) = K_t (z)$ via the requirement that 
\[
K_t = G_t^{ \langle -1 \rangle } \ \ \mbox{ (compositional 
inverse), } \ \ \forall \, t \in [0, \infty ).
\]
With these notations and in view of the calculation from 
(\ref{eqn:56d}), the statement of Theorem \ref{thm:5.4} 
amounts to checking that
\begin{equation}    \label{eqn:56f}
K(t,z) = \chitild (t,z) 
\end{equation} 
(for enough pairs $(t,z)$, with $t \in [0, \infty )$ and $z$ 
running in a neighbourhood of $0$).

The final step of this line of proof is to verify that the 
function $t \mapsto K(t,z)$ satisfies the same 
differential equation as found for $t \mapsto \chitild (t,z)$
in Equation (\ref{eqn:56c}).  This is obtained by invoking the 
partial differential equation known for $H$ from Corollary 
\ref{cor:5.2}, which expresses $\partial_t H$ as
$z( 1- 2 H \cdot \partial_z H )$.  Indeed, upon working out 
the $\partial_t$ and $\partial_z$ in the definition 
(\ref{eqn:56d2}) of $G$, one finds the said p.d.e. for $H$ to 
have the nice consequence that
\begin{equation}   \label{eqn:56g}
\partial_t G (t,z) = -2 z \, H(t,z) \, \partial_z G(t,z).
\end{equation}
So then if we take $\partial_t$ in the identity 
$G(t, K(t,z) ) = z$, we get
\begin{align*}
0 
& = \partial_t G (t, K(t,z)) +
    \partial_z G (t, K(t,z)) \, \partial_t K(t,z)           \\
& = \partial_z G (t, K(t,z)) \ \Bigl( \,
    -2 K(t,z) \, H(t, K(t,z) + \partial_t K(t,z) \, \Bigr),
\end{align*}
where at the second equality sign we made use of (\ref{eqn:56g}).
In the resulting product we are sure that
$\partial_z G (t, K(t,z)) \neq 0$ (because taking $\partial_z$
in the identity $G(t, K(t,z) ) = z$ gives 
$\partial_z G (t, K(t,z)) \cdot \partial_z K(t,z) = 1$); so 
we can divide it out, and conclude that
\begin{equation}   \label{eqn:56h}
\partial_t K(t,z)  = 2 K(t,z) \, H(t, K(t,z)).
\end{equation}
We are left to observe that 
\begin{align*}
4 \bigl( \, H(t, K(t,z))^2 - K(t,z) \, \bigr) 
& = \bigl( \, G( t, K(t,z)) + 1 \, \bigr)^2 
  \ \ \mbox{ (by def. of $G$) }                        \\
& = (z+1)^2;
\end{align*}
this implies 
$H(t, K(t,z)) = \bigl( \, K(t,z) + (z+1)^2/4 \, \bigr)^{1/2}$, 
hence (\ref{eqn:56h}) is precisely the differential equation 
which had been sought for $t \mapsto K(t,z)$.  The required 
initial condition $K(0,z) = (1 - (z+1)^2)/4$ is also easily 
verified, by writing $K_0 = G_0^{ \langle -1 \rangle }$
with $G_0 (z) = (1-4z)^{1/2} -1$.
\end{remark}

$\ $

$\ $

\section{Behaviour when $t \to \infty$}

\setcounter{equation}{0}

In this section we look at the behaviour of a joint cumulant 
$\kappa_n ( u_t^{ \omega (1) }, \ldots ,
u_t^{ \omega (n) } \, \bigr)$, for general
$\omega \in \{ 1,* \}^n$, when $t \to \infty$.  Specifically, 
we discuss how the limit and derivative at $\infty$ relate to 
the corresponding polynomial $Z_{\omega}$.

When it comes to just taking a plain limit $t \to \infty$,
things are straightforward: we have
\begin{equation}  \label{eqn:6a}
\lim_{t \to \infty}  \kappa_n 
\bigl( \, u_t^{ \omega (1)}, \ldots ,
u_t^{ \omega (n) } \, \bigr)
=  \kappa_n 
\bigl( \, u^{ \omega (1) }, \ldots , 
u^{ \omega (n) } \, \bigr),
\end{equation}
where $u$ is a Haar unitary; and the $*$-cumulants of a Haar 
unitary have a very nice form, first found in \cite{S1998}, 
which puts the spotlight on alternating strings of even length.
For later perusal throughout the section, it is convenient to 
include the latter concept into the following definition.

$\ $

\begin{definition}   \label{def:6.1}
$1^o$ Let $n$ be an even positive integer.  A string 
$\omega \in \{ 1,* \}^n$ is said to be {\em alternating} if it 
is equal either to $(1,*,1,*, \ldots , 1,*)$ or to
$(*,1,*,1, \ldots , *,1)$.

\vspace{6pt}

$2^o$ Let $n$ be an odd positive integer.  A string 
$\omega \in \{ 1,* \}^n$ is said to be {\em alternating} if it 
is obtained by a cyclic permutation of components from either 
$(1,*,1, \ldots , *,1)$ or $(*,1,* \ldots , 1,*)$.

\vspace{10pt}

\noindent
[So note that we only have $2$ alternating strings of length $n$ 
when $n$ is even, but we have $2n$ alternating strings of length 
$n$ when $n$ is odd.  A concrete example: 
\[
(1,1,*,1,*), \ (*,1,1,*,1), \ (1,*,1,1,*), \ (*,1,*,1,1), \
(1,*,1,*,1)
\]
are $5$ of the $10$ alternating strings of length $5$, and the 
remaining alternating strings of length $5$ are obtained by 
swapping the roles of `$1$' and `$*$' in the above list.]
\end{definition}

$\ $

\begin{proposition}   \label{prop:6.2}
Let $\omega = ( \omega (1), \ldots , \omega (n))$ be in 
$\{ 1,* \}^n$, for $n \in \bN$.

\vspace{6pt}

$1^o$ We have 
\begin{equation}  \label{eqn:62a}
\lim_{t \to \infty}  \kappa_n 
\bigl( \, u_t^{ \omega (1) }, \ldots ,
u_t^{ \omega (n) } \, \bigr)
= \left\{   \begin{array}{cl}
(-1)^{k-1} C_{k-1},  & \mbox{ if $n$ is even, $n = 2k$, and }      \\
                     & \mbox{ $\omega$ is an alternating string,}  \\
                     &                                             \\
0,                   & \mbox{ otherwise, }
\end{array}   \right.
\end{equation} 
with $C_{k-1}$ the $(k-1)$-th Catalan number (same as 
in Equation (\ref{eqn:1e}) of the introduction).

\vspace{6pt}

$2^o$ Suppose that $n$ is even, $n = 2k$. Consider the 
polynomial $Z_{\omega}$ and as written in Proposition \ref{prop:3.1},
\[
Z_{\omega} (x,y) = Z_{\omega}^{(2k)} (x) \, y^{2k} 
+ Z_{\omega}^{(2k-2)} (x) \, y^{2k-2} + \cdots 
+ Z_{\omega}^{(2)} (x) \, y^2 + Z_{\omega}^{(0)} (x).
\]
Then $Z_{\omega}^{(0)} (x)$ is a constant polynomial, where 
the constant is given by the right-hand side of Equation 
(\ref{eqn:62a}).
\end{proposition}

\begin{proof} $1^o$ This is the limit from (\ref{eqn:6a}),
where we also invoke the explicit formula for the $*$-cumulants 
of a Haar unitary that was found in \cite{S1998}.  (See Section 
3.4 of \cite{S1998}, or Proposition 15.1 in 
the monograph \cite{NS2006}.)

\vspace{6pt}

$2^o$ In view of how $Z_{\omega}$ is defined, from $1^o$ it 
follows that $\lim_{t \to \infty} Z_{\omega} (t, e^{-t/2} )$ 
exists.  Since $\lim_{t \to \infty} 
Z_{\omega}^{(n-2j)} (t) \cdot (e^{-t/2})^{n-2j} = 0$ 
for $j < n/2$, we then infer that 
$\lim_{t \to \infty} Z_{\omega}^{(0)} (t)$ exists as well.  But 
this can only happen if $Z_{\omega}^{(0)}$ is a constant (and the 
constant in question must be the one appearing
on the right-hand side of (\ref{eqn:62a})).
\end{proof}

$\ $

For the present paper it is important that upon looking 
at strings of odd length, we get the following analogue of 
Proposition \ref{prop:6.2}.2.

$\ $

\begin{theorem}  \label{thm:6.3}
Let $\omega = ( \omega (1), \ldots , \omega (n))$ be in 
$\{ 1,* \}^n$.  Suppose that $n$ is odd, $n = 2k-1$, and
consider the polynomial $Z_{\omega}$, written in the form 
\[
Z_{\omega} (x,y) = Z_{\omega}^{(2k-1)} (x) \, y^{2k-1} 
+ Z_{\omega}^{(2k-3)} (x) \, y^{2k-3} + \cdots 
+ Z_{\omega}^{(3)} (x) \, y^3 + Z_{\omega}^{(1)} (x) \, y.
\]
Then $Z_{\omega}^{(1)} (x)$ is a constant polynomial, and more
precisely:
\begin{equation}   \label{eqn:63a}
Z_{\omega}^{(1)} (x) =
\left\{   \begin{array}{cl}
(-1)^{k-1} C_{k-1},  & \mbox{ if $\omega$ is alternating,}    \\
                     &                                        \\
0,                   & \mbox{ otherwise. }
\end{array}   \right.
\end{equation}
\end{theorem}

\begin{proof}  
It is easily seen that every $\omega \in \{ 1,* \}^n$ has 
$\switch ( \omega ) \leq n-1$, with equality holding if and only 
if $\omega$ is alternating.  Thus if $\omega$ is not 
alternating, then Theorem \ref{thm:3.7} can be applied with
$j = (n-1)/2$, and gives that
$Z_{\omega}^{(1)} (x)$ is constantly equal to $0$.

We are left to prove the following:
\begin{equation}   \label{eqn:63b}
\left\{  \begin{array}{l}
\mbox{if $n = 2k -1$ and if 
          $\omega \in \{ 1, * \}^n$ is alternating, }  \\
\mbox{then $Z_{\omega}^{(1)} (x)$ 
is constantly equal to $(-1)^{k-1} C_{k-1}$. } 
\end{array}  \right.
\end{equation}
We will prove this statement by induction on $k$.  The case 
$k=1$ is clear, since $Z_{(1)} (x,y) = Z_{(*)} (x,y) = y$
(corresponding to the fact that $u_t$ has expectation
$e^{-t/2}$, $\forall \, t \in [0, \infty )$).  The remaining 
part of the proof is devoted to the induction step: we fix 
$k \geq 2$, we assume that (\ref{eqn:63b}) holds for 
alternating strings of length $1,3, \ldots , 2k-3$, and we 
prove that it also holds for alternating strings of length 
$2k-1$.

Since any two alternating strings of length $2k-1$ can 
be obtained from each other by operations which do not affect 
$Z_{\omega}$'s, it will suffice to verify that, for the $k$ 
that was fixed, we have
\begin{equation}  \label{eqn:63c}
Z_{ ( \, \underbrace{1,1,*, \ldots , 1,*}_{2k-1} \, )}^{(1)} (x) 
= (-1)^{k-1} C_{k-1}.
\end{equation}

In order to verify (\ref{eqn:63c}), we invoke the recursion 
from Proposition \ref{prop:3.4} (used for the string 
$\omega = (1,1,*, \ldots , 1,*) \in \{ 1,* \}^{2k-1}$), and 
we extract the coefficient of $y$ on both sides of that 
recursion.  On the right-hand side of the resulting equation
we get a sum which (same as in Equation (\ref{eqn:34a}) of 
Proposition \ref{prop:3.4}) is indexed by $m$, with 
$1 \leq m \leq n-1 = 2k-2$.  By grouping the terms of the
sum according to the parity of $m$, we obtain that
\begin{equation}  \label{eqn:63d}
Z_{( \, \underbrace{1,1,*, \ldots , 1,*}_{2k-1} \, )}^{(1)} 
= - \bigl( \, \Sigma_{\mathrm{odd}} 
            + \Sigma_{\mathrm{even}} \, \bigr),
\end{equation}
where
\[
\Sigma_{\mathrm{odd}} = 
Z_{(1)}^{(1)} 
\, Z_{( \, \underbrace{1,*, \ldots ,1,*}_{2k-2} \, )}^{(0)}
+ Z_{(1,1,*)}^{(1)} 
\, Z_{( \, \underbrace{1,*, \ldots ,1,*}_{2k-4} \, )}^{(0)}
+ \cdots +
\, Z_{( \, \underbrace{1,1,*, \ldots ,1,*}_{2k-3} \, )}^{(0)}
\, Z_{(1,*)}^{(0)}
\]
and
\[
\Sigma_{\mathrm{even}} = 
Z_{(1,1)}^{(0)} 
\, Z_{( \, \underbrace{*,1,*, \ldots ,1,*}_{2k-3} \, )}^{(1)}
+ Z_{(1,1,*,1)}^{(0)} 
\, Z_{( \, \underbrace{*,1,*, \ldots ,1,*}_{2k-5} \, )}^{(1)}
+ \cdots +
\, Z_{( \, \underbrace{1,1,*,1 \ldots ,*,1}_{2k-2} \, )}^{(0)}
\, Z_{(*)}^{(1)}.
\]
The sum $\Sigma_{\mathrm{even}}$ is equal to $0$, because of
\[
Z_{(1,1)}^{(0)} = Z_{(1,1,*,1)}^{(0)} = \cdots
= Z_{( \, \underbrace{1,1,*,1 \ldots ,*,1}_{2k-2} \, )}^{(0)} = 0
\]
(cf.~Proposition \ref{prop:6.2}.2, case of non-alternating strings).
On the other hand, the induction hypothesis and the case of 
alternating strings in Proposition \ref{prop:6.2}.2 give us that
\[
\Sigma_{\mathrm{odd}} = 
(-1)^{0} C_{0} \cdot (-1)^{k-2} C_{k-2} 
+ (-1)^{1} C_{1} \cdot (-1)^{k-3} C_{k-3} + \cdots
+ (-1)^{k-2} C_{k-2} \cdot (-1)^0 C_0.
\]
Thus Equation (\ref{eqn:63d}) comes, after all, to
\[
Z_{( \, \underbrace{1,1,*, \ldots , 1,*}_{2k-1} )}^{(1)} (x) 
= (-1)^{k-1} \sum_{j=0}^{k-2} C_j \cdot C_{k-2-j}.
\]
A basic recursion for Catalan numbers says that the 
sum on the right-hand side of the latter equality is just 
$C_{k-1}$, and the required formula (\ref{eqn:63c}) follows.
\end{proof}

$\ $

We can now follow the same kind of connection as in Proposition 
\ref{prop:6.2} (but going in reverse) in order to obtain 
the ``derivative at $t= \infty$'' for $*$-cumulants of the $u_t$'s.

$\ $

\begin{corollary}  \label{cor:6.4}
Let $n$ be a positive integer and let 
$\omega = ( \omega (1), \ldots , \omega (n))$ be a string in 
$\{ 1,* \}^n$.  We consider the limit
\begin{equation}   \label{eqn:68a}
\lim_{t \to \infty}  \frac{ 
\kappa_n \bigl( u_t^{ \omega (1) }, \ldots ,
                u_t^{ \omega (n) }  \bigr)     -
\kappa_n \bigl( u^{ \omega (1) }, \ldots ,
                u^{ \omega (n) }  \bigr)    }{e^{-t/2}}
\end{equation}
where $u_t$ is the free unitary Brownian motion at time $t$, and
$u$ is a Haar unitary.  This limit exists, and is equal to 
\begin{equation}  \label{eqn:68b}
\left\{   \begin{array}{cl}
(-1)^{k-1} C_{k-1},  & \mbox{ if $n$ is odd, $n = 2k-1$, and }     \\
                     & \mbox{ $\omega$ is an alternating string,}  \\
                     &                                             \\
0,                   & \mbox{ otherwise. }
\end{array}   \right.
\end{equation} 
\end{corollary}

\begin{proof} If $n$ is even, $n = 2k$, then the difference
on the numerator of the fraction in (\ref{eqn:68a}) is
\[
Z_{\omega}^{(2k)} (t) \cdot (e^{-t/2})^{2k} + 
Z_{\omega}^{(2k-2)} (t) \cdot (e^{-t/2})^{2k-2} + \cdots 
+ Z_{\omega}^{(2)} (t) \cdot (e^{-t/2})^{2},
\]
and when divided by $e^{-t/2}$ this is sure to go to $0$ as 
$t \to \infty$.

If $n$ is odd, $n = 2k-1$, then the difference
on the numerator of the fraction in (\ref{eqn:68a}) is
\[
Z_{\omega}^{(2k-1)} (t) \cdot (e^{-t/2})^{2k-1} + 
Z_{\omega}^{(2k-3)} (t) \cdot (e^{-t/2})^{2k-3} + \cdots 
+ Z_{\omega}^{(3)} (t) \cdot (e^{-t/2})^{3}
+ Z_{\omega}^{(1)} (t) \cdot (e^{-t/2}).
\]
When divided by $e^{-t/2}$ this converges to the constant 
$Z_{\omega}^{(1)} (t)$ described in Theorem \ref{thm:6.3}, and
the result follows.
\end{proof}

$\ $

\begin{remark}   \label{rem:6.5}
The limit from Corollary \ref{cor:6.4} points towards an 
``infinitesimal structure'' which accompanies the $*$-distribution 
of a Haar unitary, in the sense of the paper of Belinschi and 
Shlyakhtenko \cite{BS2010}.  In order to relate to the framework 
of \cite{BS2010}, one has to do a change of variable: consider the 
noncommutative probability spaces $( \cB_s , \psi_s )$, defined 
for $s \in [0,1]$, where for $s \neq 0$ we put 
$\cB_s = \cA_{-2 \log s}$, $\psi_s = \varphi_{-2 \log s}$, while 
for $s=0$ we take $( \cB_0 , \psi_0 )$ to be the space where the Haar 
unitary lives.  With this change of variable, the limit from 
(\ref{eqn:68a}) becomes a derivative at $0$, as prescribed in 
\cite{BS2010}.
\end{remark}

$\ $

\begin{remark}  \label{rem:6.6}
As reviewed in Remark \ref{rem:2.6}, the Haar unitary is a basic 
example of $R$-diagonal element, with determining sequence 
consisting of signed Catalan numbers.  For a Haar unitary $u$, 
Corollary \ref{cor:6.4} brings into the picture an additional 
``infinitesimal determining sequence'', which happens to also 
consist of signed Catalan numbers, and which determines the 
derivatives at $\infty$ of all joint cumulants of $u_t$ and 
$u_t^{*}$ (with $u_t$ seen as an approximation of $u$).  It is 
natural to extend this concept of infinitesimal determining 
sequence to the case of an $R$-diagonal element $a = uq$ as 
appearing in Remark \ref{rem:2.6}.  Indeed, we can approximate 
such an $a$ with elements of the form $a_t := u_t q$ (where $q$ 
is now assumed to also be free from $\{ u_t, u_t^{*} \}$), and we 
can consider the same kind of limits as in Equation (\ref{eqn:68a}) 
of Corollary \ref{cor:6.4}, but now in connection to $a$ and $a_t$.  
This leads to the next proposition, which provides a nice
infinitesimal analogue for the facts reviewed in Remark 
\ref{rem:2.6}.
\end{remark}

$\ $

\begin{proposition}  \label{prop:6.7}
Let $a$ and $\{ a_t \mid t \in [0, \infty ) \}$ be as in the 
preceding remark.  There exists a sequence 
$(\beta_k )_{k=1}^{\infty}$ such that for 
$\omega = (\omega (1), \ldots , \omega (n)) \in \{ 1,* \}^n$
one has:
\begin{equation}  \label{eqn:67a}
\lim_{t \to \infty}  \frac{ 
\kappa_n \bigl( a_t^{ \omega (1) }, \ldots ,
                a_t^{ \omega (n) }  \bigr)     -
\kappa_n \bigl( a^{ \omega (1) }, \ldots ,
                a^{ \omega (n) }  \bigr)    }{e^{-t/2}}
= \left\{  \begin{array}{cl}
\beta_{(n+1)/2},
    &  \mbox{if $n$ is odd and $\omega$ }            \\
    &  \mbox{ is alternating,}                       \\
0,  &  \mbox{otherwise.}
\end{array}  \right.
\end{equation} 
The $\beta_k$'s can be written in terms of the free cumulants of 
$q$ and $q^2$ via a formula similar to 
Equation (\ref{eqn:26b}), as follows:
\begin{equation}    \label{eqn:67b}
\beta_k = \sum_{\pi \in NC(k)}
\Bigl( \moeb ( \pi, 1_k) \cdot
\prod_{V \in \pi}  \kappa_{|V|} 
\bigl( \, (q^2, \ldots , q^2,q) \mid V \, \bigr) \, \Bigr),
\ \ k \in \bN .
\end{equation}
\end{proposition}

\noindent
[A concrete example: for $\omega = (1,*,1)$ the above
proposition says that
\[
\lim_{t \to \infty} \frac{\kappa_3 (u_t q, q u_t^{*}, u_t q)
- \kappa_3 (uq, qu^{*}, uq) }{e^{-t/2}} = \beta_2,
\mbox{ with $\beta_2 = - \kappa_1 (q^2) \kappa_1 (q) 
+ \kappa_2 (q^2, q)$.] }
\]

$\ $

The remaining part of this section is devoted to discussing the 
proof of Proposition \ref{prop:6.7}.  The arguments revolve
around a certain set of non-crossing partitions 
$NC_{\omega} (2n) \subseteq NC(2n)$ that is associated to an
$\omega \in \{ 1,* \}^n$.  The sets $NC_{\omega} (2n)$ are 
introduced next, and their relevance for the limits on the 
left-hand side of Equation (\ref{eqn:67a}) is explained in 
Lemma \ref{lemma:6.10} below.  The conclusion of 
Proposition \ref{prop:6.7} will then be derived via a calculation 
which relies on the structure of these $NC_{\omega} (2n)$'s. 
In order to not make the discussion excessively long, we will 
merely state (in Lemmas \ref{lemma:6.11} and \ref{lemma:6.13}) the 
relevant facts we need about $NC_{\omega} (2n)$, and we will leave 
the proofs of these purely combinatorial facts as an exercise to the 
interested reader.

$\ $

\begin{definition}   \label{def:6.8}
Consider a string
$\omega = ( \omega (1), \ldots , \omega (n)) \in \{ 1, * \}^n$.

$1^o$ We denote
$U_{\omega} := \{ 2i-1 \mid 1 \leq i \leq n, \ \omega (i) = 1 \}
\cup \{ 2i   \mid 1 \leq i \leq n, \ \omega (i) = * \}$, and 

\noindent
$Q_{\omega} :=  \{ 1, 2, \ldots , 2n \} \setminus U_{\omega}
            = \{ 2i   \mid 1 \leq i \leq n, \ \omega (i) = 1 \}
       \cup \{ 2i-1   \mid 1 \leq i \leq n, \ \omega (i) = * \}$.

\noindent
(Thus $U_{\omega}, Q_{\omega} \subseteq \{ 1, \ldots , 2n \}$, 
and they have $n$ elements each.)

\vspace{6pt}

$2^o$ We will use the notation $NC_{\omega} (2n)$ for the set of
partitions $\tau \in NC(2n)$ which fulfill all of the following 
conditions (i)--(v).

\vspace{6pt}

\noindent
(i) $\tau \vee \{ \, \{ 1,2 \}, \{ 3,4 \} , 
              \ldots , \{ 2n-1,2n \} \, \} = 1_{2n}$.

\vspace{6pt}

\noindent
(ii) For every $V \in \tau$ we have that either 
$V \subseteq Q_{\omega}$ or $V \subseteq U_{\omega}$.

\vspace{6pt}

\noindent
(iii) There exists precisely one block $V_o$ of $\tau$
such that $V_o \subseteq U_{\omega}$ and $|V_o|$ is odd.

\vspace{6pt}

\noindent
(iv) If $V_o = \{ i_1, \ldots, i_p \}$ 
(with $i_1 < \cdots < i_p$) is as in (iii) then, modulo a 
cyclic permutation, the numbers  
$i_1, \ldots , i_p$ have alternating parities.

\vspace{6pt}

\noindent
(v) If $V = \{ j_1, \ldots, j_r \}$ 
(with $j_1 < \cdots < j_r$) is a block of $\tau$ such that 
$V \subseteq U_{\omega}$ and $r$ is even, then 
$j_1, \ldots , j_r$ have alternating parities.

\vspace{6pt}

\noindent
[Note: the meaning of (iv) is that if in 
$(i_1, \ldots , i_p)$ we replace every $i_h$ which is odd by 
a ``$1$'' and every $i_h$ which is even by a ``$*$'', then 
we get an alternating string in $\{ 1,* \}^p$, in the sense 
of Definition \ref{def:6.1}.2.  The same happens for (v),
but there we don't need to mention the possibility of a 
cyclic permutation.]
\end{definition}

$\ $

\begin{example}   \label{example:6.9}
To illustrate the above terminology, let us
pursue the case when $n=3$ and $\omega = (1,*,1)$.  Then 
$U_{\omega} = \{ 1,4,5 \}$ and $Q_{\omega} = \{ 2,3,6 \}$.
Direct inspection shows that for a partition 
$\tau \in NC_{\omega} (6)$, the restriction $\tau \mid U_{\omega}$ 
has to be one of $\{ \, \{ 1,4,5 \} \, \}$ or
$\{ \, \{ 1 \}, \, \{ 4,5 \} \, \}$.  (If we try to make 
$\tau \mid U_{\omega} = \{ \, \{ 1,4 \}, \, \{ 5 \} \, \}$ then 
condition (i) of Definition \ref{def:6.8}.2 cannot be satisfied.
Likewise, trying to make $\tau \mid U_{\omega}$ be one of
$\{ \, \{ 1,5 \}, \, \{ 4 \} \, \}$ or
$\{ \, \{ 1 \}, \, \{ 4 \}, \, \{ 5 \} \, \}$ violates
condition (v), respectively (iii).)  We find in this way that 
$NC_{\omega} (6)$ consists of 5 partitions, depicted as follows.
\\
\vskip20pt
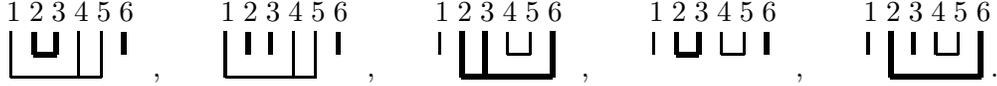
\begin{center}
  \setlength{\unitlength}{0.3cm}

\begin{picture}(6,2)
  \thicklines
  \put(1,0){\line(0,1){2}}
  \put(1,0){\line(1,0){4}}
  \put(4,0){\line(0,1){2}}
  \put(5,0){\line(0,1){2}}
  \linethickness{0.6mm}
  \put(2,1){\line(0,1){1}}
  \put(2,1){\line(1,0){1}}
  \put(3,1){\line(0,1){1}}
  \put(6,1){\line(0,1){1}}
  \put(0.8,2.5){1}
  \put(1.8,2.5){2}
  \put(2.8,2.5){3}
  \put(3.8,2.5){4}
  \put(4.8,2.5){5}
  \put(5.8,2.5){6}
\end{picture} $\ $ , \hspace{0.3cm}
\begin{picture}(6,2)
  \thicklines
  \put(1,0){\line(0,1){2}}
  \put(1,0){\line(1,0){4}}
  \put(4,0){\line(0,1){2}}
  \put(5,0){\line(0,1){2}}
  \linethickness{0.6mm}
  \put(2,1){\line(0,1){1}}
  \put(3,1){\line(0,1){1}}
  \put(6,1){\line(0,1){1}}
  \put(0.8,2.5){1}
  \put(1.8,2.5){2}
  \put(2.8,2.5){3}
  \put(3.8,2.5){4}
  \put(4.8,2.5){5}
  \put(5.8,2.5){6}
\end{picture} $\ $ , \hspace{0.3cm}
\begin{picture}(6,2)
  \thicklines
  \put(1,1){\line(0,1){1}}
  \put(4,1){\line(0,1){1}}
  \put(4,1){\line(1,0){1}}
  \put(5,1){\line(0,1){1}}
  \linethickness{0.6mm}
  \put(2,0){\line(0,1){2}}
  \put(2,0){\line(1,0){4}}
  \put(3,0){\line(0,1){2}}
  \put(6,0){\line(0,1){2}}
  \put(0.8,2.5){1}
  \put(1.8,2.5){2}
  \put(2.8,2.5){3}
  \put(3.8,2.5){4}
  \put(4.8,2.5){5}
  \put(5.8,2.5){6}
\end{picture} $\ $ , \hspace{0.3cm}
\begin{picture}(6,2)
  \thicklines
  \put(1,1){\line(0,1){1}}
  \put(4,1){\line(0,1){1}}
  \put(4,1){\line(1,0){1}}
  \put(5,1){\line(0,1){1}}
  \linethickness{0.6mm}
  \put(2,1){\line(0,1){1}}
  \put(2,1){\line(1,0){1}}
  \put(3,1){\line(0,1){1}}
  \put(6,1){\line(0,1){1}}
  \put(0.8,2.5){1}
  \put(1.8,2.5){2}
  \put(2.8,2.5){3}
  \put(3.8,2.5){4}
  \put(4.8,2.5){5}
  \put(5.8,2.5){6}
\end{picture} $\ $ , \hspace{0.3cm}
\begin{picture}(6,2)
  \thicklines
  \put(1,1){\line(0,1){1}}
  \put(4,1){\line(0,1){1}}
  \put(4,1){\line(1,0){1}}
  \put(5,1){\line(0,1){1}}
  \linethickness{0.6mm}
  \put(2,0){\line(0,1){2}}
  \put(2,0){\line(1,0){4}}
  \put(3,1){\line(0,1){1}}
  \put(6,0){\line(0,1){2}}
  \put(0.8,2.5){1}
  \put(1.8,2.5){2}
  \put(2.8,2.5){3}
  \put(3.8,2.5){4}
  \put(4.8,2.5){5}
  \put(5.8,2.5){6}
\end{picture}  .
\end{center}

\end{example}

$\ $

\begin{lemma}    \label{lemma:6.10}
Let $a$ and $\{ a_t \mid t \in [0, \infty ) \}$ be as in 
Proposition \ref{prop:6.7}, and let $\omega$ be a string in 
$\{ 1, * \}^n$.  The limit considered on the left-hand side 
of Equation (\ref{eqn:67a}) exists, and is equal to 
\begin{equation}   \label{eqn:610a}
\sum_{\tau \in NC_{\omega} (2n)} 
\term^{(U)}_{\tau} \cdot \term^{(Q)}_{\tau},
\end{equation}
where for every $\tau \in NC_{\omega} (2n)$ the numbers 
$\term^{(U)}_{\tau}$ and $\term^{(Q)}_{\tau}$ are defined 
as follows:

\noindent
$\bullet$ Let $V_o$ be the unique block of $\tau$
such that $V_o \subseteq U_{\omega}$ and $| V_o |$ is odd, and
let $V_1, \ldots , V_k$ be the other blocks of $\tau$ which 
are contained in $U_{\omega}$.  Then 
\begin{equation}   \label{eqn:610b}
\term_{\tau}^{(U)} = (-1)^{(|V_o|-1)/2 } C_{(|V_o|-1)/2} \cdot
\prod_{i=1}^k (-1)^{(|V_i|-2)/2 } C_{(|V_i|-2)/2}.
\end{equation}

\noindent
$\bullet$ Let $W_1, \ldots , W_{\ell}$ be the blocks of $\tau$ 
which are contained in $Q_{\omega}$.  Then 
\begin{equation}   \label{eqn:610c}
\term_{\tau}^{(Q)} = 
\prod_{j=1}^{\ell} \kappa_{|W_j|} (q, \ldots , q).
\end{equation}

\vspace{6pt}

\noindent
In the case (which may occur) when $NC_{\omega} (2n) = \emptyset$, 
the quantity (\ref{eqn:610a}) should be read as $0$.
\end{lemma}

\begin{proof}
In the cumulant 
$\kappa_n ( a_t^{(\omega (1) )}, \ldots , a_t^{ (\omega (n) )} )$
we replace every $a_t$ by $u_t q$ and every $a_t^{*}$ by 
$q u_t^{*}$, and we invoke the formula for a cumulant with products 
of entries which was reviewed in Remark \ref{rem:2.4}.2.  This gives
\begin{equation}    \label{eqn:610d}
\kappa_n ( a_t^{(\omega (1) )}, \ldots , a_t^{ (\omega (n) )} )
= \sum_{ \begin{array}{c}
{\scriptstyle  \tau \in NC(2n) \ with}  \\
{\scriptstyle  \tau \vee 
               \{ \{ 1,2 \}, \{ 3,4 \}, \ldots \} = 1_{2n} }
\end{array} }  \ \ \Term_{\tau},
\end{equation}
where every $\Term_{\tau}$ is a product of cumulants with entries
from $\{ q, u_t, u_t^{*} \}$.  But $q$ is free from 
$\{ u_t, u_t^{*} \}$; hence free cumulants which mix $q$ with
$\{ u_t, u_t^{*} \}$ vanish, and this implies that on the 
right-hand side of (\ref{eqn:610d}) we can restrict the sum to 
the smaller set
\[
\cT_{\omega} := \{ \tau \in NC(2n) \mid 
\tau \mbox{ fulfills conditions (i) and (ii) in 
Definition \ref{def:6.8}.2} \} .
\]

For every $\tau \in \cT_{\omega}$, the quantity $\Term_{\tau}$ 
appearing in (\ref{eqn:610d}) is a product where some factors 
are joint cumulants of $u_t$ and $u_t^{*}$, while some other 
factors are cumulants of $q$.  The dependence on $t$ is coming 
exclusively from the factors involving $u_t$ and $u_t^{*}$, which 
are quasi-polynomials in $t/2$, in the way found earlier in the 
paper.  If we are interested in the limit on the left-hand side 
of Equation (\ref{eqn:67a}), then what we have to do is pick the 
coefficient of $e^{-t/2}$ in every $\Term_{\tau}$.  (Note that 
the coefficient of $( e^{-t/2} )^0$ in $\Term_{\tau}$, if 
existing, will be removed by the subtraction of 
$\kappa_n ( a^{(\omega (1) )}, \ldots , a^{ (\omega (n) )} )$
in the numerator of (\ref{eqn:67a}).)  By invoking Proposition
\ref{prop:6.2}.2 and Theorem \ref{thm:6.3}, one easily sees that 
a partition $\tau \in \cT_{\omega}$ can include a contribution of 
order $e^{-t/2}$ in $\Term_{\tau}$ only if $\tau$ also satisfies
the conditions (iii)--(v) listed in Definition 
\ref{def:6.8}.2, i.e.~only if $\tau \in NC_{\omega} (2n)$. 
(If $NC_{\omega} (2n) = \emptyset$, then we see at this point
that the limit in (\ref{eqn:67a}) is equal to $0$.)  Finally, 
for every $\tau \in NC_{\omega} (2n)$ one has 
$\Term_{\tau} = \term^{(U)}_{\tau} \cdot \term^{(Q)}_{\tau}$, 
with $\term^{(U)}_{\tau}$ and $\term^{(Q)}_{\tau}$ described 
as in Equations (\ref{eqn:610b}), (\ref{eqn:610c}); this 
verification is immediate (with the signed Catalan numbers 
in (\ref{eqn:610b}) coming from Proposition \ref{prop:6.2}.2 
and Theorem \ref{thm:6.3}), and is left as exercise to the 
reader.
\end{proof}

$\ $

Now, there are many strings $\omega \in \{1,*\}^n$ 
with $NC_{\omega} (2n) = \emptyset$.  An obvious necessary condition
for $NC_{\omega} (2n)$ being non-empty is that 
$| \ell - \ell ' | = 1$, where 
\[
\ell := | \{ 1 \leq i \leq n \mid \omega (i) = 1 \} |
\mbox{ and }
\ell ' := | \{ 1 \leq i \leq n \mid \omega (i) = * \} | ;
\]
indeed, if it is not true that $| \ell - \ell ' | = 1$, then no
partition $\tau \in NC(2n)$ can satisfy conditions (iii)--(v) 
of Definition \ref{def:6.8}.2.  But even when 
$| \ell - \ell ' | = 1$, it still turns out that
$NC_{\omega} (2n) = \emptyset$ unless $\omega$ is alternating.
This is caused by the condition (i) of Definition \ref{def:6.8}.2.
The next lemma records the precise statement that we will need
later on; the proof of the lemma (which goes in the same spirit 
as those of Propositions 11.25 or 15.1 in \cite{NS2006}) is left 
as exercise.

$\ $

\begin{lemma}   \label{lemma:6.11}
Let $\omega = ( \omega (1), \ldots , \omega (n))$ be a string
in $\{ 1,* \}^n$ for some $n \geq 2$, such that 
$\omega (1) = \omega (n) = 1$.  If 
$NC_{\omega} (2n) \neq \emptyset$, then $n$ is odd and $\omega$
is the alternating string $(1,*,1, \ldots , *,1)$.
\hfill $\square$
\end{lemma}

$\ $

We are thus prompted to focus on alternating strings of odd length. 
In order to describe what is going on in this case, we introduce 
some additional bits of notation.

$\ $

\begin{rem-and-notation}    \label{rem:6.12}
Let $k$ be a positive integer, and consider the alternating 
string 
$\omega_k := (1,*,1, \ldots , *,1) \in \{ 1,* \}^{2k-1}$.
Note that the sets 
$U_{\omega_k}, Q_{\omega_k} \subseteq \{ 1, \ldots , 4k-2 \}$ 
associated to $\omega_k$ in Definition \ref{def:6.8}.2 are
\[
U_{\omega_k} = \{ 1,4,5,8,9, \ldots , 4k-4, 4k-3 \}
\mbox{ and }
Q_{\omega_k} = \{ 2,3,6,7, \ldots , 4k-6, 4k-5, 4k-2 \} .
\]

$1^o$ For every partition $\pi \in NC(k)$ we denote by $\pifatodd$
the partition of $U_{\omega_k}$ which is defined by ``converting''
the points $1,2, \ldots , k$ into the groups of points $\{ 1 \}$, 
then $\{ 4,5 \}$, then $\{ 8,9 \} , \ldots \, ,$ then 
$\{ 4k-4, 4k-3 \} $ of $U_{\omega_k}$.
That is, $\pifatodd$ has blocks of the form 
\[
\widetilde{V} := \bigcup_{i \in V} \{ 4i-4, 4i-3 \} , 
\mbox{ where $V$ is a block of $\pi$ such that $1 \not\in V$,}
\]
and also has a block 
\[
\widetilde{V_o} := \{ 1 \} \cup \Bigl( \, 
\bigcup_{ \begin{array}{c}
{\scriptstyle i \in V_o,} \\
{\scriptstyle i \neq 1} 
\end{array}  } \ \{ 4i-4, 4i-3 \} \, \Bigr) , 
\]
where $V_o$ is the block of $\pi$ such that $1 \in V_o$. 

Likewise, for every $\rho \in NC(k)$ we denote by $\rhofateven$
the partition of $Q_{\omega_k}$ which is defined by converting
the points $1,2, \ldots , k$ into the groups of points $\{ 2,3 \}$, 
then $\{ 6,7 \} , \ldots \, ,$ then $\{ 4k-6, 4k-5 \}$, then 
$\{ 4k-2 \}$ of $Q_{\omega_k}$.  For example, for $k=4$ we have:

$\ $
\\
\begin{center}
  \setlength{\unitlength}{0.3cm}
$\pi$ = 
  \begin{picture}(4,2)
  \thicklines
  \put(1,-1){\line(0,1){2}}
  \put(1,-1){\line(1,0){2}}
  \put(2, 0){\line(0,1){1}}
  \put(3,-1){\line(0,1){2}}
  \put(4, 0){\line(0,1){1}}
  \put(0.8,1.2){1}
  \put(1.8,1.2){2}
  \put(2.8,1.2){3}
  \put(3.8,1.2){4}
  \end{picture}
$\ $ \hspace{0.2cm} $\Rightarrow \ \pifatodd$ = 
  \begin{picture}(14,2)
  \thicklines
  \put(1,-1){\line(0,1){2}}
  \put(1,-1){\line(1,0){8}}
  \put(4, 0){\line(0,1){1}}
  \put(4, 0){\line(1,0){1}}
  \put(5, 0){\line(0,1){1}}
  \put(8,-1){\line(0,1){2}}
  \put(9,-1){\line(0,1){2}}
  \put(12,-1){\line(0,1){2}}
  \put(12,-1){\line(1,0){1}}
  \put(13,-1){\line(0,1){2}}
  \put(0.8,1.2){1}
  \put(3.8,1.2){4}
  \put(4.8,1.2){5}
  \put(7.8,1.2){8}
  \put(8.8,1.2){9}
  \put(11,1.2){12}
  \put(12.7,1.2){13}
  \end{picture}
\end{center}

\vskip6pt

and 

\vskip6pt

\begin{center}
  \setlength{\unitlength}{0.3cm}
$\rho$ = 
  \begin{picture}(3,2)
  \thicklines
  \linethickness{0.6mm}
  \put(1,-1){\line(0,1){2}}
  \put(1,-1){\line(1,0){1}}
  \put(2,-1){\line(0,1){2}}
  \put(3,-1){\line(0,1){2}}
  \put(3,-1){\line(1,0){1}}
  \put(4,-1){\line(0,1){2}}
  \put(0.8,1.2){1}
  \put(1.8,1.2){2}
  \put(2.8,1.2){3}
  \put(3.8,1.2){4}
  \end{picture}
$\ $ \hspace{0.2cm} $\Rightarrow \ \rhofateven$ = 
\begin{picture}(14,2)
  \thicklines
  \linethickness{0.6mm}
  \put(2.3,-1){\line(0,1){2}}
  \put(2.3,-1){\line(1,0){5}}
  \put(3.3,-1){\line(0,1){2}}
  \put(6.3,-1){\line(0,1){2}}
  \put(7.3,-1){\line(0,1){2}}
  \put(10.3,-1){\line(0,1){2}}
  \put(10.3,-1){\line(1,0){4}}
  \put(11.3,-1){\line(0,1){2}}
  \put(14.3,-1){\line(0,1){2}}
  \put(1.8,1.2){2}
  \put(2.8,1.2){3}
  \put(5.8,1.2){6}
  \put(6.8,1.2){7}
  \put(9.4,1.2){10}
  \put(11,1.2){11}
  \put(13.8,1.2){14}
\end{picture}
\end{center}

$\ $

$\ $

$2^o$ There exists an analogy between the notation for
$\pifatodd$ and $\rhofateven$ that was just introduced,
and the notation for $\piodd$ and $\roeven$ used in the 
description of the Kreweras complementation map, in Notation 
\ref{def:2.1}.3.  This is due to the fact that if we think of 
the sets 
\[
\{ 1 \}, \, \{ 2,3 \}, \{ 4,5 \}, \ldots , 
\{ 4k-4, 4k-3 \}, \, \{ 4k-2 \}
\]
as of a sequence of $2k$ consecutive ``fat points'', then 
$U_{\omega_k}$ covers the odd positions and $Q_{\omega_k}$ covers 
the even positions among these fat points.

As a consequence of the above, it is easily seen that if we 
start with a partition $\pi \in NC(k)$ and we draw the combined
pictures of $\pifatodd$ and $( \Kr (\pi ) )^{ \qpoints }$, 
then we obtain a partition in $NC(4k-2)$.  Moreover, 
$( \Kr (\pi ) )^{ \qpoints }$ can be characterized as the 
largest (with respected to reverse refinement order) partition 
$\sigma$ of the set $Q_{\omega_k}$ which has the property that
$\pifatodd \sqcup \sigma \in NC(4k-2)$.

\vspace{6pt}

$3^o$ Let $\rho$ be in $NC(k)$, and consider the partition
$\rhofateven$ of $Q_{\omega_k}$.  We will need to work with 
a special way of breaking the blocks of $\rhofateven$ into 
pairs (plus a singleton), which is described as follows.
Let $W = \{ j_1, \ldots , j_p \}$ be a block of $\rho$, 
where $j_1 < \cdots < j_p$.  We distinguish two cases.

\vspace{6pt}

\noindent
{\em Case 1.} $j_p \neq k$.
In this case, the block of $\rhofateven$ that corresponds 
to $W$ is
\[
\widetilde{W} := \{ 4j_1 - 2, 4j_1 -1, 
4j_2 - 2, 4j_2 -1, \ldots , 4j_p - 2, 4j_p -1 \} ,
\]
and we break it into $p$ pairs by going cyclically as follows:
\begin{equation}    \label{eqn:612a}
\{ 4j_1 - 1, 4j_2 - 2 \}, \, 
\{ 4j_2 - 1, 4j_3 - 2 \}, \ldots , 
\{4j_{p-1} - 1, 4j_p - 2 \} , \,
\{4j_1 - 2, 4j_p - 1 \} .
\end{equation}

\vspace{6pt}

\noindent
{\em Case 2.} $j_p = k$.
In this case, the block of $\rhofateven$ that corresponds 
to $W$ is
\[
\widetilde{W} := \{ 4j_1 - 2, 4j_1 -1, 
4j_2 - 2, 4j_2 -1, \ldots , 4j_{p-1} - 2, 4j_{p-1} -1, 4j_p-2 \} ,
\]
and we break it into $p-1$ pairs and a singleton as follows:
\begin{equation}    \label{eqn:612b}
\{ 4j_1 - 2 \} , \, \{ 4j_1 - 1, 4j_2 - 2 \}, \, 
\{ 4j_2 - 1, 4j_3 - 2 \}, \ldots , \{ 4j_{p-1} - 1, 4j_p - 2 \} . 
\end{equation}

\vspace{6pt}

The partial pairing (with one singleton block) of $Q_{\omega_k}$
which results upon doing all the breaking described in 
(\ref{eqn:612a}) and (\ref{eqn:612b}) will be denoted as 
$\rhoqpairing$.  In the example with $k=4$ depicted in part 
$1^o$ of this notation, we get

$\ $
\\
\begin{center}
  \setlength{\unitlength}{0.3cm}
$\rho$ = 
  \begin{picture}(3,2)
  \thicklines
  \linethickness{0.6mm}
  \put(1,-1){\line(0,1){2}}
  \put(1,-1){\line(1,0){1}}
  \put(2,-1){\line(0,1){2}}
  \put(3,-1){\line(0,1){2}}
  \put(3,-1){\line(1,0){1}}
  \put(4,-1){\line(0,1){2}}
  \put(0.8,1.2){1}
  \put(1.8,1.2){2}
  \put(2.8,1.2){3}
  \put(3.8,1.2){4}
  \end{picture}
$\ $ \hspace{0.2cm} $\Rightarrow \ \rhoqpairing$ = 
\begin{picture}(14,2)
  \thicklines
  \linethickness{0.6mm}
  \put(2,-1){\line(0,1){2}}
  \put(2,-1){\line(1,0){5}}
  \put(3, 0){\line(0,1){1}}
  \put(3, 0){\line(1,0){3}}
  \put(6, 0){\line(0,1){1}}
  \put(7,-1){\line(0,1){2}}
  \put(10, 0){\line(0,1){1}}
  \put(11,-1){\line(0,1){2}}
  \put(11,-1){\line(1,0){3}}
  \put(14,-1){\line(0,1){2}}
  \put(1.8,1.2){2}
  \put(2.8,1.2){3}
  \put(5.8,1.2){6}
  \put(6.8,1.2){7}
  \put(9.2,1.2){10}
  \put(10.6,1.2){11}
  \put(13.6,1.2){14}
\end{picture}
\end{center}
\end{rem-and-notation}

$\ $

$\ $

Based on the notation introduced above we describe, in the 
next lemma, the structure of a general partition in 
$NC_{\omega_k} (4k-2)$.  (The statement of the lemma also uses
the lattice structure of $NC( Q_{\omega_k} )$, and invokes the
``$\sqcup$'' operation --- this is analogous to the note recorded
in Remark \ref{rem:2.1}.)

$\ $

\begin{lemma}   \label{lemma:6.13}
Let $k$ be a positive integer, let $\omega_k$ be the 
alternating string
$(1,*,1 \ldots , *,1) \in \{ 1,* \}^{2k-1}$, and consider 
the terminology introduced by Notation \ref{rem:6.12} in 
connection to $\omega_k$.

\vspace{6pt}

$1^o$ Let $\tau$ be a partition in $NC_{\omega_k} (4k-2)$.  
Then the restricted partition $\tau \mid U_{\omega_k}$ is of 
the form $\pifatodd$ for a (uniquely determined) $\pi \in NC(k)$.
The quantity $\term_{\tau}^{(U)}$ appearing in (\ref{eqn:610b})
of Lemma \ref{lemma:6.10} is precisely equal to 
$\moeb ( 0_k, \pi )$ for this $\pi \in NC(k)$.

\vspace{6pt}

$2^o$ Let $\tau$ be a partition in $NC_{\omega_k} (4k-2)$, 
let $\pi \in NC(k)$ be such that 
$\tau \mid U_{\omega_k} = \pifatodd$, and let us denote 
$\Kr ( \pi ) =: \rho \in NC(k)$.  Then the partition 
$\sigma := \tau \mid Q_{\omega_k} \in NC( Q_{\omega_k} )$
has the properties that
$\sigma \leq \rhofateven$ and 
$\sigma \vee \rhoqpairing = \rhofateven$.

\vspace{6pt}

$3^o$ Conversely: let $\pi$ be in $NC(k)$, put 
$\rho := \Kr ( \pi )$, and let $\sigma$ be a partition in  
$NC( Q_{\omega_k} )$ with the properties that
$\sigma \leq \rhofateven$ and
$\sigma \vee \rhoqpairing = \rhofateven$.  Then the 
partition $\tau := \pifatodd \sqcup \sigma$ of 
$\{ 1, \ldots , 4k-2 \}$ is in $NC_{\omega_k} (4k-2)$.
\hfill  $\square$
\end{lemma}

$\ $

\noindent
{\bf Proof of Proposition \ref{prop:6.7}.} 
If $n$ is even, then it is immediate that 
$NC_{\omega} (2n) = \emptyset$ for every $\omega \in \{ 1,* \}^n$
(cf.~discussion preceding Lemma \ref{lemma:6.11}).  Thus in 
this case the limit on the left-hand side of Equation 
(\ref{eqn:67a}) is indeed equal to $0$, as required.  
For the rest of the proof we will assume that $n$ is odd.

\vspace{6pt}

We next note a couple of reductions that can be done on $\omega$.

(a) Observe that if the conclusion of the proposition 
holds for a string $\omega = ( \omega (1), \ldots , \omega (n))$, 
then it also holds for the ``adjoint'' string $\omega^{*}$ with 
entries
\[
\omega^{*} (i) = \left\{  \begin{array}{ll}
1, & \mbox{ if $\omega (n+1-i) = *$}   \\
*, & \mbox{ if $\omega (n+1-i) = 1$} 
\end{array}   \right\} , \ \ 1 \leq i \leq n.
\]
This follows from the left-right symmetry of free cumulants that was 
noted in part (c) of Remark \ref{rem:2.4}.3, where we also take 
into account the fact that $\{ u_t^{*} \mid t \geq 0 \}$ form a
free unitary Brownian motion, and that (due to the hypothesis 
$q=q^{*}$) all the limits postulated on the right-hand side of 
Equation (\ref{eqn:67a}) are real numbers.

(b) Observe that if the conclusion of the proposition holds for a 
string $\omega$, then it also holds for an $\omega '$ obtained by 
cyclically permuting the entries of $\omega$.  This is immediate 
from the invariance property of free cumulants noted in part (a) 
of Remark \ref{rem:2.4}.3.

\vspace{6pt}

As a result of the above observation (a), we see that it suffices 
to handle strings $\omega \in \{ 1,* \}^n$ that have
$\mathopen| \{  1 \leq i \leq n \mid \omega (i) = 1 \} \mathclose|
>
\mathopen| \{  1 \leq i \leq n \mid \omega (i) = * \} \mathclose|$.
By doing a suitable cyclic permutation of entries for such an 
$\omega$ and by using observation (b), we then see that it 
suffices to prove the proposition under the additional assumption 
that $\omega (1) = \omega (n) = 1$.

\vspace{6pt}

If $\omega (1) = \omega (n) = 1$ but $\omega$ is not alternating,
then Lemma \ref{lemma:6.11} says that $NC_{\omega} (2n) = \emptyset$.
The limit on the left-hand side of (\ref{eqn:67a}) is therefore 
equal to $0$, as required.

\vspace{6pt}

We are left to discuss the case when $\omega$ is alternating of 
odd length, with $\omega (1) = \omega (n) = 1$.  This is exactly the
case when 
$\omega = \omega_k = (1,*,1, \ldots , *,1) \in \{ 1,* \}^{2k-1}$
for some $k \in \bN$, as discussed in Notation \ref{rem:6.12}.  In
the remaining part of the proof we fix $k$, and we will prove that
for $\omega = \omega_k$, the limit on the left-hand side of 
(\ref{eqn:67a}) is equal to the $\beta_k$ described in Equation 
(\ref{eqn:67b}).  In view of Lemma \ref{lemma:6.10}, it will suffice 
to verify the equality between the summation formulas that appear 
in (\ref{eqn:610a}) and on the right-hand side of Equation 
(\ref{eqn:67b}).  The sum from (\ref{eqn:610a}) takes here the form 
\begin{equation}   \label{eqn:614a}
\sum_{\tau \in NC_{\omega_k} (4k-2)} 
\term^{(U)}_{\tau} \cdot \term^{(Q)}_{\tau}.
\end{equation}
But a partition $\tau \in NC_{\omega_k} (4k-2)$ is 
parametrized in Lemma \ref{lemma:6.13} in terms of a pair 
$( \pi , \sigma )$ with $\pi \in NC(k)$ and 
$\sigma \in NC( Q_{\omega_k} )$; recalling the specifics of how 
that goes (which includes the writing of $\term^{(U)}_{\tau}$ as 
$\moeb (0_k, \pi )$ and the writing of $\term^{(Q)}_{\tau}$ in terms 
of $\sigma$), we see that (\ref{eqn:614a}) can be rewritten in 
the form
\[
\sum_{ \begin{array}{c}  
        {\scriptstyle \pi \in NC(k)}  \\
        {\scriptstyle with \ \Kr ( \pi ) =: \rho} 
       \end{array} } \ \moeb (0_k, \pi ) \, \Bigl( 
\, \sum_{\begin{array}{c}
          {\scriptstyle \sigma \in NC( Q_{\omega_k} ), }  \\
          {\scriptstyle \sigma \vee \rhoqpairing = \rhofateven  }
         \end{array} } \ \prod_{W \in \sigma}
          \kappa_{|W|} (q,q, \ldots , q) \, \Bigr).
\]
In the above expression it is convenient that in the first sum 
we do the change of variable $\rho = \Kr ( \pi )$ (where we also 
substitute $\moeb (0_k, \pi )$ as $\moeb ( \rho , 1_k)$); the 
quantity in (\ref{eqn:614a}) takes then the form 
\begin{equation}   \label{eqn:614b}
\sum_{ \rho \in NC(k) } \moeb (\rho , 1_k ) \, \Bigl( 
\, \sum_{\begin{array}{c}
          {\scriptstyle \sigma \in NC( Q_{\omega_k} ), }  \\
          {\scriptstyle \sigma \vee \rhoqpairing = \rhofateven  }
         \end{array} } \ \prod_{W \in \sigma}
          \kappa_{|W|} (q,q, \ldots , q) \, \Bigr).
\end{equation}

Now let us fix for the moment a partition 
$\rho = \{ B_1, \ldots , B_{\ell} \} \in NC(k)$ and, for this 
particular $\rho$, let us examine the summation over $\sigma$ 
which appears in (\ref{eqn:614b}).  Let $W_1, \ldots , W_{\ell}$ 
be the blocks of $\rhofateven$, where $W_j$ corresponds to $B_j$ 
in the natural way (cf.~discussion in Notation \ref{rem:6.12}.3),
and where let us assume that $k \in B_{\ell}$, hence 
$4k-2 \in W_{\ell}$.  It is easy to see that the summation over 
$\sigma$ in (\ref{eqn:614b}) amounts in fact to doing $\ell$ 
independent summations, over partitions 
$\sigma_1 \in NC(W_1), \ldots , \sigma_{\ell} \in NC(W_{\ell})$.
Moreover, the formula (\ref{eqn:24b}) for cumulants with products 
as entries applies to each of these $\ell$ summations, leading to 
the conclusion that the result of the $j$-th summation is
\[
\left\{   \begin{array}{ll}
\kappa_{|W_j|} (q^2, \ldots , q^2), & \mbox{ if $j < \ell$,}   \\
\kappa_{|W_j|} 
( q, \, \underbrace{ q^2, \ldots , q^2}_{|W_j|-1} \, ) 
                                    & \mbox{ if $j = \ell$.} 
\end{array}   \right.
\]
Since $|W_j| = |B_j|$ for every $1 \leq j \leq \ell$, and since
$\kappa_{|B_{\ell}|} ( q,  q^2, \ldots , q^2 ) =
\kappa_{|B_{\ell}|} ( q^2, \ldots , q^2,q )$, we obtain that for 
our fixed $\rho$ we have:
\begin{equation}   \label{eqn:614c}
\sum_{\begin{array}{c}
        {\scriptstyle \sigma \in NC( Q_{\omega_k} ), }  \\
        {\scriptstyle \sigma \vee \rhoqpairing = \rhofateven  }
      \end{array} } \ \prod_{W \in \sigma}
         \kappa_{|W|} (q,q, \ldots , q) 
= \prod_{B \in \rho} 
\kappa_{|B|} \bigl( \, (q^2, \ldots , q^2,q) \mid B \, \bigr) .
\end{equation}

Finally, we let $\rho$ run in $NC(k)$, we substitute 
Equation (\ref{eqn:614c}) into the second summation from 
(\ref{eqn:614b}), and we arrive to the right-hand side of 
Equation (\ref{eqn:67b}), as we wanted.
\hfill $\square$

$\ $

$\ $

\textbf{Acknowledgements.}  

Part of the research presented in this paper was done while 
Guay-Paquet and Nica visited IRMAR Rennes,
in February and May of 2014.  The support and the friendly 
work environment of IRMAR are gratefully acknowledged.

We also express our thanks to the anonymous referee whose 
suggestions lead to an enriched content for Sections 5 and 6 
of the paper.

$\ $

$\ $

\end{document}